\documentclass[12pt,twoside,leqno]{article}
\usepackage[T1]{fontenc}
\usepackage{amsfonts}
\usepackage{amsmath,amsthm}
\usepackage{amssymb,latexsym}
\usepackage{enumerate}

\theoremstyle{plain}
\newtheorem{theorem}{Theorem}[section]
\newtheorem{lemma}[theorem]{Lemma}
\newtheorem{proposition}[theorem]{Proposition}
\newtheorem{corollary}[theorem]{Corollary}

\theoremstyle{definition}

\newtheorem{definition}[theorem]{Definition}
\theoremstyle{remark}
\newtheorem{remark}[theorem]{Remark}

\newtheorem{problem}[theorem]{Problem}

\DeclareMathOperator{\Iso}{{\rm Iso}}

\DeclareMathOperator{\supp}{{\rm supp}}

\DeclareMathOperator{\CB}{{\rm CB}}
\DeclareMathOperator{\cl}{{\rm cl}}

\begin{document}
\title{On iso-dense and scattered spaces in $\mathbf{ZF}$}
\author{Kyriakos Keremedis, Eleftherios Tachtsis and Eliza Wajch\\
Department of Mathematics, University of the Aegean\\
Karlovassi, Samos 83200, Greece\\
kker@aegean.gr\\
Department of Statistics and Actuarial-Financial Mathematics,\\
 University of the Aegean, Karlovassi 83200, Samos, Greece\\
 ltah@aegean.gr\\
Institute of Mathematics\\
Faculty of Exact and Natural Sciences \\
Siedlce University of Natural Sciences and Humanities\\
ul. 3 Maja 54, 08-110 Siedlce, Poland\\
eliza.wajch@wp.pl}
\maketitle
\begin{abstract}
 A topological space is iso-dense if it has a dense set of isolated points. A topological space is scattered if each of its non-empty subspaces has an isolated point. In $\mathbf{ZF}$, in the absence of the axiom of choice, basic properties of iso-dense spaces are investigated. A new permutation model is constructed in which a discrete weakly Dedekind-finite space can have the Cantor set as a remainder. A metrization theorem for a class of quasi-metric spaces is deduced.  The statement ``every compact scattered metrizable space is separable'' and several other statements about metric iso-dense spaces are shown to be equivalent to the countable axiom of choice for families of finite sets. Results concerning the problem of whether it is provable in $\mathbf{ZF}$ that every non-discrete compact metrizable space contains an infinite compact scattered subspace are also included.
  \medskip

\noindent\textit{Mathematics Subject Classification (2010)}: 03E25, 03E35,  54E35, 54G12, 54D35\newline 
\textit{Keywords}: Weak forms of the Axiom of Choice, scattered space, iso-dense space, (quasi-)metric space
\end{abstract}

\section{Preliminaries}
\label{s1}
\subsection{Set-theoretic framework and preliminary definitions}
\label{s1.1}
In this article, the intended context for reasoning and statements of theorems is $\mathbf{ZF}$ without any form of the axiom of choice $\mathbf{AC}$. However, we also refer to permutation models of $\mathbf{ZFA}$ (cf. \cite{Je} and \cite{hr}). In this article, we are concerned mainly with iso-dense and scattered spaces in $\mathbf{ZF}$, defined as follows:

\begin{definition}
\label{s1d1} 
A topological space $\mathbf{X}$ is called:
\begin{enumerate}
\item[(i)] \emph{iso-dense} if the set $\Iso(X)$ of all isolated points of $\mathbf{X}$ is dense in $\mathbf{X}$;
\item[(ii)] \emph{scattered} or \emph{dispersed} if, for every non-empty subspace $\mathbf{Y}$ of $\mathbf{X}$, $\emptyset\neq\Iso(Y)$.
\end{enumerate}
\end{definition}

Before we pass to the main body of the article, let us establish notation and recall some known definitions in this subsection, make a list of weaker forms of $\mathbf{AC}$ in Subsection \ref{s1.2}, and recall several known results for future references in Subsection \ref{s1.3}. The content of the article is described in brief in Subsection \ref{s1.4}. Our main new results are included in Sections \ref{s2}-\ref{s4}.

We denote by $ON$ the class of all (von Neumann) ordinal numbers. The first infinite ordinal number is denoted by $\omega$. Then $\mathbb{N}=\omega\setminus\{0\}$. If $X$ is a set, the power set of  $X$ is denoted by $\mathcal{P}(X)$, and the set of all finite subsets of $X$ is denoted by $[X]^{<\omega}$. 

 A \emph{quasi-metric} on a set $X$ is a function $d: X\times X\to [0, +\infty)$ such that, for all $x,y,z\in X$, $d(x,y)\leq d(x,z)+d(z,y)$ and $d(x,y)=0$ if and only if $x=y$ (cf. \cite{fl}, \cite{Kel}, \cite{kunzi}, \cite{wils}). If a quasi-metric $d$ on $X$ is such that $d(x,y)=d(y,x)$ for all $x,y\in X$, then $d$ is a \emph{metric}. A \emph{(quasi-)metric space} is an ordered pair $\langle X, d\rangle$ where $X$ is a set and $d$ is a (quasi-)metric on $X$.

Let $d$ be a quasi-metric on $X$. \textit{The conjugate} of $d$ is the quasi-metric $d^{-1}$ defined by: 
$$d^{-1}(x, y)=d(y, x) \text{ for } x, y\in X.$$
The \emph{metric} $d^{\star}$ \emph{associated with} $d$ is defined by: 
$$d^{\star}(x,y)=\max\{d(x,y), d(y,x)\} \text{ for } x,y\in X.$$  

Clearly, $d$ is a metric if and only if $d=d^{-1}=d^{\star}$.

The $d$-\textit{ball with centre $x\in X$ and radius} $r\in(0, +\infty)$ is the set 
$$B_{d}(x, r)=\{ y\in X: d(x, y)<r\}.$$
 The collection 
$$\tau(d)=\left\{ V\subseteq X: (\forall x\in V)(\exists n\in\omega) B_{d}\left(x, \frac{1}{2^n}\right)\subseteq V\right\}$$
is the \textit{topology in $X$ induced by $d$}.   For a set $A\subseteq X$, let $\delta_d(A)=0$ if $A=\emptyset$, and let $\delta_d(A)=\sup\{d(x,y): x,y\in A\}$ if $A\neq \emptyset$. Then $\delta_d(A)$ is the \emph{diameter} of $A$ in $\langle X, d\rangle$.

 A quasi-metric $d$ on $X$ is called \emph{strong} if $\tau(d)\subseteq\tau(d^{-1})$.
 
In the sequel, topological or (quasi-)metric spaces (called spaces in abbreviation) are denoted by boldface letters, and the underlying sets of the spaces are denoted by lightface letters.

For a topological space $\mathbf{X}=\langle X, \tau\rangle$ and for $Y\subseteq X$, let $\tau|_Y=\{V\cap Y: V\in\tau\}$ and let $\mathbf{Y}=\langle Y, \tau|_Y\rangle$. Then $\mathbf{Y}$ is the topological subspace of $\mathbf{X}$ such that $Y$ is the underlying set of $\mathbf{Y}$. If this is not misleading, we may denote the topological subspace $\mathbf{Y}$ of $\mathbf{X}$ by $Y$.

 A topological space $\mathbf{X}=\langle X, \tau\rangle$ is called \emph{(quasi-)metrizable} if there exists a (quasi-)metric $d$ on $X$ such that $\tau=\tau(d)$.

For a (quasi-)metric space $\mathbf{X}=\langle X, d\rangle$ and for $Y\subseteq X$, let $d_Y=d\upharpoonright Y\times Y$ and $\mathbf{Y}=\langle Y, d_Y\rangle$. Then $\mathbf{Y}$ is the (quasi-)metric subspace of $\mathbf{X}$ such that $Y$ is the underlying set of $\mathbf{Y}$. Given a (quasi-)metric space $\mathbf{X}=\langle X, d\rangle$, if not stated otherwise, we also denote by $\mathbf{X}$ the topological space $\langle X, \tau(d)\rangle$. For every $n\in\mathbb{N}$, $\mathbb{R}^n$ denotes also $\langle \mathbb{R}^{n}, d_e\rangle$ and $\langle \mathbb{R}^n, \tau(d_e)\rangle$  where $d_e$ is the Euclidean metric on $\mathbb{R}^n$.    

For a topological space $\mathbf{X}=\langle X, \tau\rangle$ and a set $A\subseteq X$, we denote by $\cl_{\mathbf{X}}(A)$ or by $\cl_{\tau}(A)$  the closure of $A$ in $\mathbf{X}$. 

For any topological space $\mathbf{X}=\langle X, \tau\rangle$, let 
$$\Iso_{\tau}(X)=\{x\in X: x\text{ is an isolated point of } \mathbf{X}\}.$$
\noindent If this is not misleading, as in Definition \ref{s1d1}, we use $\Iso(X)$ to denote $\Iso_{\tau}(X)$. 
By transitive recursion, we define a decreasing sequence $(X^{({\alpha})})_{\alpha\in ON}$ of closed subsets of $\mathbf{X}$ as follows:
$$X^{(0)}= X,$$
$$X^{({\alpha+1})}=X^{({\alpha})}\setminus\Iso(X^{({\alpha})}),$$
$$X^{({\alpha})}=\bigcap\limits_{\gamma\in\alpha}X^{({\gamma})}\text{ if $\alpha$ is a limit ordinal}.$$
\noindent For $\alpha\in ON$, the set $X^{({\alpha})}$ is called the \emph{$\alpha$-th Cantor-Bendixson derivative} of $\mathbf{X}$. The least ordinal $\alpha$ such that $X^{({\alpha+1})}=X^{({\alpha})}$ is denoted by $|\mathbf{X}|_{\CB}$ and is called the \emph{Cantor-Bendixson rank} of $\mathbf{X}$.

\begin{definition}
\label{s1d2} A set $X$ is called:
\begin{enumerate} 
\item[(i)]\emph{Dedekind-finite} if there is no injection $f:\omega\to X$; \emph{Dedekind-infinite} if $X$ is not Dedekind-finite;
\item[(ii)] \emph{quasi Dedekind-finite} if $[X]^{<\omega}$ is Dedekind-finite; \emph{quasi Dedekind-infi\-nite} if $X$ is not quasi Dedekind-finite;
\item[(iii)] \emph{ weakly Dedekind-finite} if $\mathcal{P}(X)$ is Dedekind-finite; \emph{weakly Dedekind-infinite} if $\mathcal{P}(X)$ is Dedekind-infinite;
\item[(iv)] a \emph{cuf set} if $X$ is a countable union of finite sets;
\item[(v)] \emph{amorphous} if $X$ is infinite and, for every infinite subset $Y$ of $X$, the set $X\setminus Y$ is finite. 
\end{enumerate}
\end{definition}

\begin{definition}
\label{s1d3}
\begin{enumerate}
\item[(i)] A space $\mathbf{X}$ is called a \emph{cuf space} if its underlying set $X$ is a cuf set.
\item[(ii)] A base $\mathcal{B}$ of a space $\mathbf{X}$ is called a \emph{cuf base} if $\mathcal{B}$ is a cuf set.
\end{enumerate}
\end{definition}

\begin{definition}
\label{s1d4}
A space $\mathbf{X}$ is called:
\begin{enumerate}
\item[(i)] \emph{first-countable} if every point of $\mathbf{X}$ has a countable base of open neighborhoods;
\item[(ii)] \emph{second-countable} if $\mathbf{X}$ has a countable base;
\item[(iii)] \emph{compact} if every open cover of $\mathbf{X}$ has a finite subcover;
\item[(iv)] \emph{locally compact} if every point of $\mathbf{X}$ has a compact neighborhood;
\item[(v)] \emph{limit point compact} if every infinite subset of $\mathbf{X}$ has an accumulation point in $\mathbf{X}$ (cf. \cite{k01} and \cite{k}).
\end{enumerate}
\end{definition}

\begin{definition}
\label{s1d5}
Let $\mathbf{X}=\langle X, d\rangle$ be a (quasi-)metric space.
\begin{enumerate}
\item[(i)] Given a real number $\varepsilon>0$, a subset $D$ of $X$ is called $\varepsilon$-\emph{dense} or an $\varepsilon$-\emph{net} in $\mathbf{X}$ if, for every $x\in X$, $B_d(x, \varepsilon)\cap D\neq\emptyset$ (equivalently, if $X=\bigcup\limits_{x\in D}B_{d^{-1}}(x, \varepsilon)$). 
\item[(ii)]  $\mathbf{X}$ is called \emph{precompact} (respectively, \emph{totally bounded}) if, for every real number $\varepsilon>0$, there exists a finite $\varepsilon$-net in $\langle X, d^{-1}\rangle$ (respectively, in $\langle X, d^{\star}\rangle$).

\item[(iii)]  $d$ is called  \emph{precompact} (respectively, \emph{totally bounded}) if $\mathbf{X}$ is precompact (respectively, \emph{totally bounded}).
\end{enumerate}
\end{definition}

\begin{remark}
\label{s1r6}
Definition \ref{s1d5} (ii) is based on the notions of precompact and totally bounded quasi-uniformities defined, e.g., in \cite{fl} and \cite{kunzi}.  Namely, given a quasi-metric  $d$ on a set $X$,  the collection 
$$\mathcal{U}(d)=\left\{U\subseteq X\times X: \exists_{n\in\omega} \left\{\langle x,y\rangle\in X\times X: d(x,y)<\frac{1}{2^n}\right\}\subseteq U\right\}$$
is a quasi-uniformity on $X$ called  the \emph{quasi-uniformity induced by} $d$ (cf. \cite[p. 504]{kunzi}). The quasi-uniformity $\mathcal{U}(d)$ is precompact (resp.,  totally bounded) in the sense of \cite{fl} and \cite{kunzi} if and only if $d$ is precompact (resp., totally bounded) in the sense of Definition \ref{s1d5}. Clearly $d$ is totally bounded if and only if, for every $n\in\omega$, there exists a finite set $D\subseteq X$ such that $X=\bigcup\limits_{x\in D}\left(B_d(x, \frac{1}{2^n})\cap B_{d^{-1}}(x, \frac{1}{2^n})\right)$. The notions of a totally bounded and precompact metric are equivalent.
\end{remark}

We recall that a (Hausdorff) \emph{compactification} of a space $\mathbf{X}=\langle X, \tau\rangle$ is an ordered pair $\langle\mathbf{Y},\gamma\rangle$ where $\mathbf{Y}$ is a  (Hausdorff) compact space and $\gamma: \mathbf{X}\to\mathbf{Y}$ is a homeomorphic embedding such that $\gamma(X)$ is dense in $\mathbf{Y}$. A compactification  $\langle \mathbf{Y}, \gamma\rangle$ of $\mathbf{X}$ and the space $\mathbf{Y}$ are usually denoted by $\gamma\mathbf{X}$. The underlying set of $\gamma\mathbf{X}$ is denoted by $\gamma X$. The subspace $\gamma X\setminus X$ of $\gamma \mathbf{X}$ is called the \emph{remainder} of $\gamma\mathbf{X}$. A space $\mathbf{K}$ is said to be a remainder of $\mathbf{X}$ if there exists a Hausdorff compactification $\gamma\mathbf{X}$ of $\mathbf{X}$ such that $\mathbf{K}$ is homeomorphic to $\gamma X \setminus X$. For compactifications  $\alpha\mathbf{X}$ and $\gamma\mathbf{X}$ of $\mathbf{X}$, we write $\gamma\mathbf{X}\leq\alpha\mathbf{X}$ if there exists a continuous mapping $f:\alpha\mathbf{X}\to\gamma\mathbf{X}$ such that $f\circ\alpha=\gamma$. If $\alpha\mathbf{X}$ and $\gamma\mathbf{X}$ are Hausdorff compactifications of $\mathbf{X}$ such that $\alpha\mathbf{X}\leq\gamma\mathbf{X}$ and $\gamma\mathbf{X}\leq\alpha\mathbf{X}$, then we write $\alpha\mathbf{X}\approx\gamma\mathbf{X}$ and say that the compactifications $\alpha\mathbf{X}$ and $\gamma\mathbf{X}$ are \emph{equivalent}. If $n\in\mathbb{N}$, then a compactification $\gamma\mathbf{X}$ of $\mathbf{X}$ is said to be an $n$-point compactification of $\mathbf{X}$ if $\gamma X\setminus X$ is an $n$-element set. 

\begin{definition}
\label{s1d7}
Let $\mathbf{X}=\langle X, \tau\rangle$ is a non-compact locally compact Hausdorff space and let $\mathcal{K}(X)$ be the collection of all compact subsets of $\mathbf{X}$. For an element $\infty\notin X$, we define $X(\infty)=X\cup\{\infty\}$, 
$$\tau(\infty)=\tau\cup\{X(\infty)\setminus K: K\in\mathcal{K}(X)\}$$
\noindent and $\mathbf{X}(\infty)=\langle X(\infty), \tau(\infty)\rangle$. Then $\mathbf{X}(\infty)$ is called the \emph{Alexandroff compactification} of $\mathbf{X}$.
\end{definition}

For every non-compact locally compact Hausdorff space $\mathbf{X}$, $\mathbf{X}(\infty)$ is the unique (up to $\approx$) one-point Hausdorff compactification of $\mathbf{X}$. Therefore, every one-point Hausdorff compactification of $\mathbf{X}$ is called the Alexandroff compactification of $\mathbf{X}$. Chandler's book \cite{ch} is a good introduction to Hausdorff compactifications in $\mathbf{ZFC}$.  Basic facts about Hausdorff compactifications in $\mathbf{ZF}$ can be found in \cite{kw0}. If $\mathbf{X}$ is a space which has the \v Cech-Stone compactification, then, as usual, $\beta \mathbf{X}$ stands for the \v Cech-Stone compactification of $\mathbf{X}$.

Given a collection  $\{X_j: j\in J\}$ of sets, for every $i\in J$, we denote by $\pi_i$ the projection $\pi_i:\prod\limits_{j\in J}X_j\to X_i$ defined by $\pi_i(x)=x(i)$ for each $x\in\prod\limits_{j\in J}X_j$. If $\tau_j$ is a topology in $X_j$, then $\mathbf{X}=\prod\limits_{j\in J}\mathbf{X}_j$ denotes the Tychonoff product of the topological spaces $\mathbf{X}_j=\langle X_j, \tau_j\rangle$ with $j\in J$. If $\mathbf{X}_j=\mathbf{X}$ for every $j\in J$, then $\mathbf{X}^{J}=\prod\limits_{j\in J}\mathbf{X}_j$. As in \cite{En}, for an infinite set $J$ and the unit interval $[0,1]$ of $\mathbb{R}$, the cube $[0,1]^J$ is called the \emph{Tychonoff cube}. If $J$ is denumerable, then the Tychonoff cube $[0,1]^J$ is called the \emph{Hilbert cube}. We denote by $\mathbf{2}$ the discrete space with the underlying set $2=\{0, 1\}$. If $J$ is an infinite set, the space $\mathbf{2}^J$ is called the \emph{Cantor cube}.

We recall that if $\prod\limits_{j\in J}X_j\neq\emptyset$, then it is said that the family $\{X_j: j\in J\}$ has a choice function, and every element of $\prod\limits_{j\in J}X_j$ is called a \emph{choice function} of the family $\{X_j: j\in J\}$. A \emph{multiple choice function} of $\{X_j: j\in J\}$ is every function $f\in\prod\limits_{j\in J}([X_j]^{<\omega}\setminus\{\emptyset\})$. If $J$ is infinite, a function $f$ is called \emph{partial} (\emph{multiple}) \emph{choice function} of $\{X_j: j\in J\}$ if there exists an infinite subset $I$ of $J$ such that $f$ is a (multiple) choice function of $\{X_j: j\in I\}$. Given a non-indexed family $\mathcal{A}$, we treat $\mathcal{A}$ as an indexed family $\mathcal{A}=\{x: x\in\mathcal{A}\}$ to speak about a  (partial) choice function and a (partial) multiple choice function of $\mathcal{A}$.

Let  $\{X_j: j\in J\}$ be a disjoint family of sets, that is, $X_i\cap X_j=\emptyset$ for each pair $i,j$ of distinct elements of $J$. If $\tau_j$ is a topology in $X_j$ for every $j\in J$, then $\bigoplus\limits_{j\in J}\mathbf{X}_j$ denotes the direct sum of the spaces $\mathbf{X}_j=\langle X_j, \tau_j\rangle$ with $j\in J$.

\begin{definition}
\label{s1d8}
(Cf. \cite{br}, \cite{lo} and \cite{kerta}.) 
\begin{enumerate}
\item[(i)] A space $\mathbf{X}$ is said to be \emph{Loeb} (respectively, \emph{weakly Loeb}) if the family of all non-empty closed subsets of $\mathbf{X}$ has a choice function (respectively, a multiple choice function).
\item[(ii)] If $\mathbf{X}$ is a (weakly) Loeb space, then every (multiple) choice function of the family of all non-empty closed subsets of $\mathbf{X}$ is called a (\emph{weak}) \emph{Loeb function} of $\mathbf{X}$.
\end{enumerate}
\end{definition}

Other topological notions used in this article but not defined here are standard. They can be found, for instance, in \cite{En} and \cite{w}. 

\subsection{The list of forms weaker than $\mathbf{AC}$}
\label{s1.2}
In this subsection, for readers' convenience, we define and denote the weaker forms of $\mathbf{AC}$ used directly in this paper. For the known forms given in \cite{hr}, we quote in their statements the form number under which they are recorded in \cite{hr}.

\begin{definition}
\label{s1d9}
\begin{enumerate}
\item $\mathbf{IQDI}$: Every infinite set is quasi Dedekind-infinite.
\item $\mathbf{IWDI}$ (\cite[Form 82]{hr}): Every infinite set is weakly Dedekind-infinite. 
\item $\mathbf{IDI}$ ( \cite[Form 9]{hr}): Every infinite set is Dedekind-infinite.
\item $\mathbf{CAC}$ (\cite[Form 8]{hr}): Every denumerable family of non-empty sets has a choice function.
\item $\mathbf{CAC}_{fin}$ ( \cite[Form 10]{hr}): Every denumerable family of non-empty finite sets has a choice function.
\item $\mathbf{WOAC}_{fin}$ (\cite[Form 122]{hr}): Every well-orderable non-empty family of non-empty finite sets has a choice function.
\item $\mathbf{MC}$ (the Axiom of Multiple Choice, \cite[Form 67]{hr}): Every non-empty family of non-empty sets has a multiple choice function.
\item $\mathbf{CMC}$ (the Countable Axiom of Multiple Choice, \cite[Form 126]{hr}): Every denumerable family of non-empty sets has a multiple choice function.
\item $\mathbf{WoAm}$ (\cite[Form 133]{hr}): Every infinite set is either well-orderable or has an amorphous subset. 
\item $\mathbf{DC}$ (the Principle of Dependent Choice, \cite[Form 43]{hr}): For every non-empty set $A$ and every binary relation $S$ on $A$ such that $(\forall x\in A)(\exists y\in A)(\langle x, y\rangle\in S)$, there exists $a\in A^{\omega}$ such that:
 $$(\forall n\in\omega)(\langle a(n), a(n+1)\rangle\in S).$$
\item $\mathbf{BPI}$ (the Boolean Prime Ideal Principle,  \cite[Form 14]{hr}): Every Boolean algebra has a prime ideal.
\item $\mathbf{NAS}$ ( \cite[Form 64]{hr}): There are no amorphous sets.
\item $\mathbf{M}(C, S)$: Every compact metrizable space is separable. (Cf. \cite{k02}, \cite{k01}, \cite{KT1}, \cite{KT2} and \cite{ktw0}.)
\item $\mathbf{IDFBI}$: For every infinite set $D$, the Cantor cube $\mathbf{2}^{\omega}$ is a remainder of the discrete space $\langle D, \mathcal{P}(D)\rangle$. (Cf. \cite{ktw1}.)
\item  $\mathbf{INSHC}$: Every infinite discrete space has a non-scattered Hausdorff compactification.
\end{enumerate} 
\end{definition}

The form $\mathbf{IDFBI}$ has been introduced and investigated in \cite{ktw1} recently. More comments about $\mathbf{IDFBI}$ are included in Remark \ref{s1r21}. New facts concerning $\mathbf{IDFBI}$ (among them, a solution of an open problem posed in \cite{ktw1}), are included in Section \ref{s2}. The form $\mathbf{INSHC}$ is new here. That $\mathbf{INSHC}$ is essentially weaker than $\mathbf{IDFBI}$ is shown in Section \ref{s2}.

\subsection{Some known results}
\label{s1.3}
In this subsection, we quote several known results that we refer to in the sequel. Some of the quoted results have been obtained recently, so they can be unknown to possible readers of this article. 

\begin{proposition}
\label{s1p10}
(Cf. \cite{kft}.) $(\mathbf{ZF})$  A topological space $\mathbf{X}$ is scattered if and only if there exists $\alpha\in ON$ such that $X^{({\alpha})}=\emptyset$. If $\mathbf{X}$ is scattered then 
$$|\mathbf{X}|_{\CB}=\min\{\alpha\in ON: X^{({\alpha})}=\emptyset\}.$$
Moreover, if $\mathbf{X}$ is a non-empty scattered compact space, then $|\mathbf{X}|_{\CB}$ is a successor ordinal. 
\end{proposition}

\begin{theorem}
\label{s1t11}
(Cf., e.g., \cite{CP}.) $(\mathbf{ZF})$ Every non-empty dense-in-itself compact second-countable Hausdorff space  is of size at least $|\mathbb{R}|$. 
\end{theorem}

\begin{proposition}
\label{s1p12}
(Cf. \cite[Proposition 2.1.11]{kunzi} and \cite{kw1}.) $(\mathbf{ZF})$ 
If $d$ is a quasi-metric on $X$ such that $\langle X, \tau(d)\rangle$ is compact, then $d$ is strong.
 \end{proposition}
 
 \begin{theorem} 
 \label{Umt}
(Cf. \cite{gt}.) $(\mathbf{ZF})$ 
\begin{enumerate}
\item[(i)] (Cf. \cite{gt}.) (Urysohn's Metrization Theorem) Every second-countable $T_3$-space is metrizable.
\item[(ii)] (Cf. \cite{ktw1}.) Every $T_3$-space which has a cuf base can be embedded in a metrizable Tychonoff cube and that, it is metrizable.
\item[(iii)] (Cf. \cite{ktw1}.) A $T_3$-space $\mathbf{X}$ is embeddable in a compact metrizable Tychonoff cube if and only if $\mathbf{X}$ is embeddable in the Hilbert cube $\mathbf[0, 1]^{\omega}$. 
\end{enumerate}
 \end{theorem}

Several essential applications of Theorem \ref{Umt}(ii), especially to the theory of Hausdorff compactifications in $\mathbf{ZF}$, have been shown in \cite{ktw1} recently. We show some other applications of Theorem \ref{Umt}(ii) in the forthcoming Sections \ref{s3} and \ref{s4}.

\begin{theorem}
\label{qm1}
(Cf. \cite{kw1}.) $(\mathbf{ZF})$
\begin{enumerate}  
\item[($a$)] For every compact Hausdorff, quasi metric space $\mathbf{X}=\langle X,d\rangle$
the following are equivalent:
\begin{enumerate}
\item[(i)] $\mathbf{X}$ is Loeb;
\item[(ii)] $\langle X, d^{-1}\rangle$ is separable;
\item[(iii)] $\mathbf{X}$ and $\langle X,d^{-1}\rangle$ are both separable;
\item[(iv)] $\mathbf{X}$ is second-countable.
\end{enumerate}
 In particular, every compact, Hausdorff, quasi-metrizable Loeb space is metrizable.
\item[($b$)] $\mathbf{CAC}$ implies that every compact, Hausdorff quasi-metrizable space is metrizable.
\end{enumerate}
\end{theorem}

\begin{proposition}
\label{s1p15}
(Cf. \cite{ktw1}.) $(\mathbf{ZF})$ Every weakly Loeb regular space which admits a cuf base has a dense cuf set.
\end{proposition}

 \begin{theorem}
\label{t:loeb} 
(Cf. \cite{lo}.) $(\mathbf{ZF})$ Let $\kappa $ be an infinite cardinal number of von Neumann, $\{\mathbf{X}_{i}:i\in \kappa \}$ be a family of compact spaces, $\{f_{i}:i\in \kappa \}$ be a collection of functions such
that for every $i\in \kappa ,f_{i}$ is a Loeb function of $\mathbf{X}_{i}$. Then the Tychonoff product $%
\mathbf{X}=\prod\limits_{i\in \kappa }\mathbf{X}_{i}$ is compact. 
\end{theorem}

\begin{theorem}
\label{BPIeq} (Cf., e.g., \cite[Theorem 4.37]{her} and \cite[Forms: 14, 14 A, 14 J]{hr}.) $(\mathbf{ZF})$ The following statements are equivalent to $\mathbf{BPI}$:
\begin{enumerate}
\item[(i)] For every non-empty set $X$, every filter in $\mathcal{P}(X)$ can be enlarged to an ultrafilter in $\mathcal{P}(X)$.
\item[(ii)] Every product of compact Hausdorff spaces is compact.
\end{enumerate}
\end{theorem}

\begin{theorem}
\label{BPIandNAS}
(Cf. \cite{hr} and \cite{ktw1}.) $(\mathbf{ZF})$ $\mathbf{BPI}$ implies $\mathbf{NAS}$ but this implication is not reversible.
\end{theorem}
 
\begin{theorem}
\label{s1t:main1}
(Cf. \cite{ktw1}.) $(\mathbf{ZF})$ For every locally compact Hausdorff space $\mathbf{X}$, the following conditions are all equivalent:
\begin{enumerate}
\item[(a)] every non-empty second-countable compact Hausdorff space is a remainder of $\mathbf{X}$;
\item[(b)] the Cantor cube $\mathbf{2}^{\omega}$ is a remainder of $\mathbf{X}$;
\item[(c)] there exists a family $\mathcal{V}=\{\mathcal{V}^n_i: n\in\mathbb{N}, i\in \{1,\ldots, 2^n\}\}$ such that, for every $n\in\mathbb{N}$, the following conditions are satisfied:
\begin{enumerate}
\item[(i)] for every $i\in\{1,\ldots, 2^n\}$, $\mathcal{V}_i^n$ is a non-empty family of open sets of $\mathbf{X}$ such that $\mathcal{V}_i^n$ is stable under finite unions and finite intersections, and, for every $U\in V_i^n$, the set $\cl_{\mathbf{X}}U$ is non-compact;
\item[(ii)]  for every $i\in\{1,\ldots, 2^n\}$ and for any $U, V\in\mathcal{V}_i^n$, $\cl_{\mathbf{X}}(U)\setminus V$ is compact;
\item[(iii)] for every pair $i,j$ of distinct elements of $\{1,\ldots, 2^n\}$, for any $W\in \mathcal{V}_{i}^n$ and $G\in\mathcal{V}_j^n$,  there exist $U\in \mathcal{V}_{2i-1}^{n+1}, V\in \mathcal{V}_{2i}^{n+1}$ with $\cl_{\mathbf{X}}(U\cup V)\setminus W$ compact and $\cl_{\mathbf{X}}(( U\cup V)\cap G)$ compact; 
\item[(iv)] if, for every $i\in\{1,\ldots, 2^n\}$,   $V_i\in\mathcal{V}_i^n$, then $X\setminus\bigcup\limits_{i=1}^{2^n}V_i$ is compact.
\end{enumerate}
\end{enumerate}
\end{theorem}

\begin{theorem} 
\label{s1t19}
(Cf. \cite{ktw1}.) $(\mathbf{ZF})$ For a set $D$, let  $\mathbf{D}=\langle D, \mathcal{P}(D)\rangle$. Then the following statements hold:
\begin{enumerate}
\item[(i)] If $\mathbf{D}$ is a cuf space, then every non-empty second-countable compact Hausdorff space is a remainder of a metrizable compactification of $\mathbf{D}$. In particular, all non-empty second countable compact Hausdorff spaces are remainders of metrizable compactifications of $\mathbb{N}$.
\item[(ii)] If $D$ is weakly Dedekind-infinite, then every non-empty second-count\-able compact Hausdorff space is a remainder of $\mathbf{D}$. 
\end{enumerate}
\end{theorem}

\begin{theorem}
\label{ind1}
(Cf. \cite{ktw1}.) $(\mathbf{ZF})$
\begin{enumerate}
\item[(i)] $\mathbf{IDFBI}$ implies  $\mathbf{NAS}$ but this implication cannot be reversed. 
\item[(ii)] The statement ``All non-empty metrizable compact spaces are remainders of metrizable compactifications of $\mathbb{N}$'' is equivalent to $\mathbf{M}(C, S)$ and, thus, it implies $\mathbf{CAC}_{fin}$.
\end{enumerate}
\end{theorem}

\begin{remark}
\label{s1r21}
In $\mathbf{ZFC}$, an archetype of Theorem \ref{s1t:main1} is included in \cite[Theorem 2.1]{hm0}; however, in $\cite{ktw1}$, Theorem 2.1 of \cite{hm0} has been shown to be unprovable in $\mathbf{ZF}$. In \cite{ktw1}, an infinite set $D$ is called \emph{dyadically filterbase infinite} if $\mathbf{2}^{\omega}$ is a remainder of the discrete space $\langle D, \mathcal{P}(D)\rangle$. An equivalent purely set-theoretic definition of a dyadically filterbase infinite set is given in \cite{ktw1} and it can be easily obtained from condition ($c$) of Theorem \ref{s1t:main1} applied to discrete spaces. Clearly, $\mathbf{IDFBI}$ is equivalent to the sentence ``Every infinite set is dyadically filterbase infinite''.
\end{remark}

\begin{theorem}
\label{comprem}
(Cf. \cite{kw0}.) $(\mathbf{ZF})$
\begin{enumerate}
\item[(i)] For every non-empty compact Hausdorff space $\mathbf{K}$, there exists a Dede\-kind-infinite discrete space $\mathbf{D}$ such that $\mathbf{K}$ is a remainder of $\mathbf{D}$.
\item[(ii)] If $D$ is an infinite set, then the Alexandroff compactification of the discrete space $\mathbf{D}=\langle D, \mathcal{P}(D)\rangle$ is the unique (up to the equivalence) Hausdorff compactification of $\mathbf{D}$ if and only if $D$ is amorphous.
\end{enumerate}
\end{theorem}

\begin{proposition}
\label{discommet} $(\mathbf{ZF})$ Let $D$ be an infinite set and let $\mathbf{D}=\langle D, \mathcal{P}(D)\rangle$. Then:
\begin{enumerate}
\item[(i)] $\mathbf{D}(\infty)$ is metrizable if and only if $D$ is a cuf set (cf. \cite{kw1});
\item[(ii)] the discrete space $\mathbf{D}$ has a metrizable compactification if and only if $D$ is a cuf set (cf. \cite{ktw1}).
\end{enumerate}
\end{proposition}

As we have already mentioned at the beginning of Subsection \ref{s1.1}, in the sequel, we apply not only $\mathbf{ZF}$-models but also permutation models of $\mathbf{ZFA}$. To transfer a statement $\mathbf{\Phi}$ from a permutation model to a $\mathbf{ZF}$-model, we use the Jech-Sochor First Embedding Theorem (see, e.g., \cite[Theorem 6.1]{Je}) if $\mathbf{\Phi}$ is a boundable statement. When $\mathbf{\Phi}$ has a permutation model but $\mathbf{\Phi}$ is a conjunction of statements each one of which is equivalent to $\mathbf{BPI}$ or to an injectively boundable statement, we use Pincus' transfer results (see \cite{pin1}, \cite{pin2} and \cite[Note 3, page 286]{hr}) to show that $\mathbf{\Phi}$ has a $\mathbf{ZF}$-model. The notions of boundable and injectively boundable statements can be found in \cite{pin1}, \cite[Problem 1 on page 94]{Je} and  \cite[Note 3, page 284]{hr}. Every boundable statement is equivalent to an injectively boundable one (see \cite{pin1} or \cite[Note 3, page 285]{hr}). We recommend \cite[Chapter 4]{Je} as an introduction to permutation models.

\subsection{The content concerning new results in brief}
\label{s1.4}

In Section \ref{s2}, we notice that, in $\mathbf{ZF}$, the class of all iso-dense compact Hausdorff spaces is essentially wider than the class of all Hausdorff compact scattered spaces; similarly, the class of all iso-dense compact metrizable spaces is essentially wider than the class of all compact metrizable scattered spaces.  A compact Hausdorff iso-dense space may fail to be completely regular in $\mathbf{ZF}$ (see Proposition \ref{ntych}). We show that the new form $\mathbf{INSHC}$  holds in every model of $\mathbf{ZF+BPI}$, is independent of $\mathbf{ZF}$, does not imply $\mathbf{BPI}$ and is strictly weaker than $\mathbf{IDFBI}$ (see Theorem \ref{s2t2}). We construct a new permutation model to prove that a dyadically filterbase infinite set can be weakly Dedekind-finite (see the proof to Theorem \ref{thm:3}. This solves an open problem posed by us in \cite{ktw1}. 

In Section \ref{s3},  we prove in $\mathbf{ZF}$ that if $\langle X, d\rangle$ is a quasi-metric $T_3$-space  such that $d$ is strong and either $\langle X, \tau(d)\rangle$ is limit point compact or $d^{-1}$ is precompact, then the space $\langle X, \tau(d)\rangle$ is metrizable (see Theorem \ref{s2t7}).  This result and its direct consequence that if $\langle X, d\rangle$ is compact Hausdorff quasi-metric space such that $\langle X, \tau(d^{-1})\rangle$ is iso-dense, then $\langle X, \tau(d)\rangle$ is metrizable (see Corollary \ref{s2c8}) are new applications of Theorem \ref{Umt}(ii) and adjuncts to Theorem \ref{qm1}. By applying Theorem \ref{Umt}, we show in $\mathbf{ZF}$ that if  $\langle X, d\rangle$ is an iso-dense metric space such that either $d$ is totally bounded or $\langle X, \tau(d)\rangle$ is limit point compact, then $\langle X, \tau(d)\rangle$ has a cuf base and can be embedded in a metrizable Tychonoff cube (see Theorem \ref{s2t10}). 

Section \ref{s4} concerns equivalents of $\mathbf{CAC}_{fin}$ in terms of scattered or iso-dense spaces (see Theorems \ref{s3t02} and \ref{s3c05}). Among our new equivalents of $\mathbf{CAC}_{fin}$ there are, for instance, the following sentences: (a) for every iso-dense metric space $\mathbf{X}$, if $\mathbf{X}$ is either limit point compact or totally bounded, then $\mathbf{X}$ is separable; (b) every totally bounded scattered metric space is countable; (c) every compact metrizable scattered space is countable; (d) every totally bounded, complete scattered metric space is compact. We show that, in $\mathbf{ZF}$, every compact metrizable cuf space is scattered (see Theorem \ref{s3p06}). We prove that $\mathbf{WOAC}_{fin}$ is equivalent to the sentence:   for every well-orderable non-empty set $S$ and every family $\{\langle X_s, d_s\rangle: s\in S\}$ of compact scattered metric spaces, the product $\prod\limits_{s\in S}\langle X_s, \tau(d_s)\rangle$ is compact (see Theorem \ref{woac}). Several other relevant results are included in Section \ref{s4}, too.

Section \ref{s5} concerns the problem of whether it is provable in $\mathbf{ZF}$ that every non-empty dense-in-itself compact metrizable space contains an infinite compact scattered subspace. Among other results of Section \ref{s4}, we show the following; (a)  each of $\mathbf{IDI}$, $\mathbf{WoAm}$ and $\mathbf{BPI}$  implies that every infinite compact first-countable Hausdorff space contains a copy of $\mathbb{N}(\infty)$; (b) every infinite first-countable compact Hausdorff separable space contains a copy of $\mathbb{N}(\infty)$; (c) every infinite first-countable compact Hausdorff space having an infinite cuf subset contains a copy of $\mathbf{D}(\infty)$ for some infinite discrete cuf space (see Theorem \ref{s4t2}).  We prove that the sentence ``every infinite first-countable Hausdorff compact space contains an infinite metrizable compact scattered subspace'' implies neither $\mathbf{CAC}_{fin}$ nor $\mathbf{IQDI}$, nor $\mathbf{CMC}$ in $\mathbf{ZFA}$ (see Theorem \ref{s4t04}).

Section \ref{s6} contains a shortlist of new open problems strictly relevant to the topic of this paper.

\section{From compact Hausdorff iso-dense spaces\newline that are not scattered to $\mathbf{INSHC}$}
\label{s2}

Since every isolated point of an open subspace of a topological space $\mathbf{X}$ is an isolated point of $\mathbf{X}$, it is obvious that the following proposition holds in $\mathbf{ZF}$:

\begin{proposition}
\label{s2p1}
$(\mathbf{ZF})$ Every scattered space is iso-dense.
\end{proposition}

Let us notice that a compact Hausdorff space is iso-dense if and only if it is a Hausdorff compactification of a discrete space. Every iso-dense locally compact Hausdorff space which satisfies condition (c) of Theorem \ref{s1t:main1} has an iso-dense non-scattered Hausdorff compactification. In particular, for every dyadically filterbase infinite set $D$, the discrete space $\langle D, \mathcal{P}(D)\rangle$ has a non-scattered Hausdorff compactification.   
 It follows from Theorem \ref{s1t19}(i) that every denumerable discrete space has non-scattered metrizable Hausdorff compactifications. Thus, in $\mathbf{ZF}$, the class of all (compact) metrizable scattered spaces is essentially smaller than the class of all iso-dense (compact) metrizable spaces, and the class of all (compact Hausdorff) scattered spaces is essentially smaller than the class of all (compact Hausdorff) iso-dense spaces. 
 
 It was proved in \cite{kft} that it holds in $\mathbf{ZF}$ that every compact Hausdorff scattered space is zero-dimensional, so completely regular. Let us show that a compact Hausdorff iso-dense space may fail to be completely regular in $\mathbf{ZF}$. 

\begin{proposition}
\label{ntych}
There exists a model $\mathcal{M}$ of $\mathbf{ZF}$ in which there is a compact Hausdorff iso-dense space which is not completely regular.
\end{proposition}
\begin{proof}
Let $\mathcal{M}$ be any model of $\mathbf{ZF}$ in which there exists a compact Hausdorff, not completely regular space $\mathbf{Z}$ (cf, e.g., \cite{gt}) and let us work inside $\mathcal{M}$. By Theorem \ref{comprem}(i), it holds in $\mathcal{M}$ that there exists a Hausdorff compactification $\gamma\mathbf{D}$ of a discrete space $\mathbf{D}$ such that $\gamma D\setminus D$ is homeomorphic to $\mathbf{Z}$. Then, in $\mathcal{M}$,  $\gamma\mathbf{D}$ is a non-scattered, iso-dense compact Hausdorff space which is not completely regular.
\end{proof}

Let us shed more light on the forms $\mathbf{INSHC}$ and $\mathbf{IDFBI}$. 

\begin{theorem}
\label{s2t2}
$(\mathbf{ZF})$
\begin{enumerate} 
\item[(i)] Every compact Hausdorff space is a subspace of a compact Hausdorff iso-dense space.
\item[(ii)] Every compact second-countable Hausdorff space is a subspace of a compact second-countable iso-dense space.
\item[(iii)]  $\mathbf{DC}\rightarrow\mathbf{IWDI}\rightarrow\mathbf{IDFBI}\rightarrow\mathbf{INSHC}\rightarrow\mathbf{NAS}$; 
\item[(iv)] $\mathbf{BPI}\rightarrow\mathbf{INSHC}$;
\item[(v)] $\mathbf{INSHC}\nrightarrow \mathbf{IDFBI}$ and $\mathbf{INSHC}\nrightarrow\mathbf{BPI}$.
\end{enumerate}
\end{theorem}
\begin{proof}
To show that (i) holds, it suffices to apply Theorem \ref{comprem}. It follows directly from Theorem \ref{s1t19}(i) that (ii) holds. 

(iii) It is known that $\mathbf{DC}$ implies $\mathbf{IWDI}$ (see, e.eg., \cite[pages 326 and 339]{hr}). It has been noticed in \cite{ktw1} that, by  Theorem \ref{s1t19}(ii), $\mathbf{IWDI}$ implies $\mathbf{IDFBI}$. The implications $\mathbf{IDFBI}\rightarrow\mathbf{INSHC}\rightarrow\mathbf{NAS}$ can be deduced from Theorems \ref{s1t:main1} and \ref{comprem}(ii). 

To prove (iv), we assume $\mathbf{BPI}$ and consider any infinite set $D$. Let $\mathbf{D}=\langle D, \mathcal{P}(D)\rangle$. In the light of Theorem \ref{BPIeq}(i) and \cite[Theorem 3.27]{kw0}, it follows from $\mathbf{BPI}$ that there exists the \v Cech-Stone compactification $\beta\mathbf{D}$ of $\mathbf{D}$. Suppose that $\beta D\setminus D$ has an isolated point $y_0$. Then there exist disjoint open subsets $U, V$ of $\beta\mathbf{D}$ such that $y_0\in U$, $(\beta D\setminus D)\setminus\{y_0\}\subseteq  V$ and $U\cap V=\emptyset$. Then the subspace $U\cup\{y_0\}$ of $\beta\mathbf{D}$ is the \v Cech-Stone compactification of the subspace $U\cap D$ of $\mathbf{D}$. It follows from Theorem \ref{comprem}(ii) that $U\cap D$ is amorphous but this is impossible because $\mathbf{BPI}$ implies $\mathbf{NAS}$ by Theorem \ref{BPIandNAS}. The contradiction obtained shows that $\beta D\setminus D$ is dense-in-itself, so $\beta\mathbf{D}$ is not scattered. 

(v) It was shown in \cite{ktw1} that the conjunction $\mathbf{BPI}\wedge\neg\mathbf{IDFBI}$ has a $\mathbf{ZF}$-model. This, together, with (iv), implies that there is a model of $\mathbf{ZF}$ in which the conjunction $\mathbf{INSHC}\wedge\neg\mathbf{IDFBI}$ is true. Hence $\mathbf{INSHC}\nrightarrow\mathbf{IDFBI}$. 

To prove $\mathbf{INSHC}\nrightarrow\mathbf{BPI}$, let us use the Feferman's forcing model $\mathcal{M}2$ in \cite{hr}. It is known that $\mathbf{DC}\wedge\neg\mathbf{BPI}$ is true in $\mathcal{M}2$ (see \cite[page 148]{hr}). To complete the proof, it suffices to notice that it follows from (iii) that $\mathbf{INSHC}$ is also true in $\mathcal{M}2$.
\end{proof}

\begin{remark}
\label{s2r3}
($a$) We do not know if the conjunction $\mathbf{NAS}\wedge \neg\mathbf{INSHC}$ has a $\mathbf{ZF}$-model. 

($b$) It is worth noticing that it follows from    Theorem \ref{comprem}(ii) that it holds in $\mathbf{ZF}$ that $\mathbf{NAS}$ is equivalent to the following statements:
\begin{enumerate}
\item[(i)] Every infinite discrete space has a two-point Hausdorff compactification.
\item[(ii)] For every natural number $n$, every infinite discrete space has an $n$-point Hausdorff compactification. 
\end{enumerate}

($c$) In view of Theorem \ref{s2t2}(iii), it holds in $\mathbf{ZF}$ that $\mathbf{INSHC}$ follows from every form of \cite{hr} which implies $\mathbf{IWDI}$. In particular, the implication $\mathbf{CMC}\rightarrow\mathbf{INSHC}$ holds in $\mathbf{ZF}$ (see \cite[page 339]{hr}).

($d$) It is known that $\mathbf{BPI}$ implies $\mathbf{CAC}_{fin}$ (see, e.g., \cite[pages 325 and 354]{hr}). Since $\mathbf{BPI}$ implies $\mathbf{INSHC}$ by Theorem \ref{s2t2}(ii), it is worth noticing that the conjunction $\mathbf{INSHC}\wedge\neg\mathbf{CAC}_{fin}$ has a $\mathbf{ZF}$-model. To see this, let us notice that, in the Second Fraenkel Model $\mathcal{N}2$ in \cite{hr}, the conjunction $\mathbf{IWDI}\wedge\neg\mathbf{CAC}_{fin}$ is true (see \cite[page 179]{hr}). Since $\mathbf{IWDI}\wedge\neg\mathbf{CAC}_{fin}$ is a conjunction of two injectively boundable statements and it has a permutation model, it also has a $\mathbf{ZF}$-model by Pincus' transfer theorems. This, together with Theorem \ref{s2t2}(iii), implies that $\mathbf{INSHC}\wedge\neg\mathbf{CAC}_{fin}$ has a $\mathbf{ZF}$-model. This is also an alternative proof that $\mathbf{INSHC}\wedge\neg\mathbf{BPI}$ has a $\mathbf{ZF}$-model.
\end{remark}

It has been shown in the proof to Theorem \ref{s2t2} that $\mathbf{INSHC}\wedge\neg\mathbf{IDFBI}$ has a $\mathbf{ZF}$. Thus, by Theorem \ref{s2t2}(iii), $\mathbf{INSHC}\wedge\neg\mathbf{IWDI}$ has a $\mathbf{ZF}$-model. Let us recall the following open problems posed by us in \cite{ktw1}:
 
\begin{problem}
\label{s2prob:5}
\begin{enumerate}
\item[(i)] Is there a $\mathbf{ZF}$-model for $\mathbf{IDFBI}\wedge\neg\mathbf{IWDI}$? (See \cite[Problem (3) of Section 6]{ktw1}.)
\item[(ii)] Is there a model of $\mathbf{ZF}$ in which a weakly Dedekind-finite set can be dyadically filterbase infinite? (See \cite[Problem (6) of Section 6]{ktw1}.)
\end{enumerate}
\end{problem}

In \cite[the proof to Theorem 5.14]{ktw1}, a permutation model has been constructed in which there exists a weakly Dedekind-finite discrete space which has a remainder homeomorphic to $\mathbb{N}(\infty)$. Now, we are in a position to solve Problem \ref{s2prob:5}(ii) (that is, Problem (6) from Section 6 in \cite{ktw1}) by the following theorem:

\begin{theorem}
\label{thm:3}
It is relatively consistent with $\mathbf{ZF}$ that there exists a dyadically filterbase infinite set which is weakly Dedekind-finite. 
\end{theorem}
\begin{proof}
Let $\mathbf{\Phi}$ be the following statement: ``There exists a dyadically filterbase infinite set which is weakly Dedekind-finite''.

Since $\mathbf{\Phi}$ is a boundable statement, by the Jech--Sochor First Embedding Theorem (see \cite[Theorem 6.1]{Je}), it suffices to prove that $\mathbf{\Phi}$ has a permutation model.  To this end, let us modify the model constructed in \cite[the proof to Theorem 5.14]{ktw1} to get a new permutation model $\mathcal{N}$ in which $\mathbf{\Phi}$ is true. 

In what follows, for an arbitrary non-empty set $S$ and every permutation $\psi$ of $S$, we denote by $\supp(\psi)$ the support of $\psi$, that is, $\supp(\psi)=\{x\in S: \psi(x)\neq x\}$.

We start with a model $M$ of $\mathbf{ZFA}+\mathbf{AC}$ with a denumerable set $A$ of atoms such that $A$ has a denumerable partition $\mathcal{A}=\{A_{i}:i\in\omega\}$ into infinite sets. In $M$, we let 
$$\mathcal{B}=\{\mathcal{B}_{i}^{n}:n\in\mathbb{N},i\in\{1,2,\ldots, 2^{n}\}\}$$
be a family with the following two properties:

\begin{enumerate}
\item[($a$)] For $n=1$, $\{\mathcal{B}_{1}^{1},\mathcal{B}_{2}^{1}\}$ is a partition of $\mathcal{A}=\{A_{i}:i\in\omega\}$ into two infinite sets.
\item[($b$)] For every $n\in\mathbb{N}$ and for every $i\in\{1,2,\ldots,2^{n}\}$, $\{\mathcal{B}_{2i-1}^{n+1},\mathcal{B}_{2i}^{n+1}\}$ is a partition of $\mathcal{B}_{i}^{n}$ into two infinite sets. 
\end{enumerate}
We may thus view $\mathcal{B}$ as an infinite binary tree, having $\mathcal{A}$ as its root.

Let $\mathcal{G}$ be the group of all permutations $\phi$ of $A$ which satisfy the following two properties:

\begin{enumerate}
\item[($c$)] $\phi$ moves only finitely many elements of $A$.
\item[($d$)] $(\forall i\in\omega)(\exists j\in\omega)(\exists F\in [A_{j}]^{<\omega})(\phi[\supp(\phi\upharpoonright A_{i})]=F)$. 
\end{enumerate}

For every $n\in\mathbb{N}$ and for every $i\in\{1,2,\ldots,2^{n}\}$, we let  
$$\mathcal{Q}_{i}^{n}=\{\bigcup\{\phi(Z):Z\in\mathcal{B}_{i}^{n}\}:\phi\in \mathcal{G}\}.$$
We also let
$$\mathcal{Q}=\bigcup\{\mathcal{Q}_{i}^{n}:n\in\mathbb{N},i\in\{1,2,\ldots,2^{n}\}\}.$$
 
\noindent For every $E\in [\mathcal{Q}]^{<\omega}$, we let 
$$\mathcal{G}_{E}=\{\phi\in \mathcal{G}:{\forall Q\in E}(\phi(Q)=Q)\}.$$ Then $\mathcal{G}_E$ is a subgroup of $\mathcal{G}$. Furthermore, since for all $E,E^{\prime}\in [\mathcal{Q}]^{<\omega}$, $\mathcal{G}_{E\cup E^{\prime}}\subseteq\mathcal{G}_E\cap\mathcal{G}_{E^{\prime}}$, the collection $\{\mathcal{G}_E: E\in [\mathcal{Q}]^{<\omega}\}$ is a  base for a filter in the set of all subgroups of $\mathcal{G}$. Let $\mathcal{F}$ be the filter of subgroups of $\mathcal{G}$ generated by $\{\mathcal{G}_E: E\in [\mathcal{Q}]^{<\omega}\}$. To check that $\mathcal{F}$ is a normal filter on $\mathcal{G}$, we need to show that $\mathcal{F}$ has the following two properties:
\begin{equation}
\label{eq:normalf1}
\forall a\in A(\{\pi\in \mathcal{G}:\pi(a)=a\}\in\mathcal{F})
\end{equation}
and
\begin{equation}
\label{eq:normalf2}
(\forall\pi\in \mathcal{G})(\forall H\in\mathcal{F})(\pi H\pi^{-1}\in\mathcal{F}).
\end{equation}
\smallskip

To argue for (\ref{eq:normalf1}), let $a\in A$. Since $\mathcal{A}$ is a partition of $A$ in $M$, there exists a unique $i\in\omega$ such that $a\in A_{i}$. Since the set $\{\mathcal{B}_{1}^{1},\mathcal{B}_{2}^{1}\}$ is a partition of $\mathcal{A}$, either $A_{i}\in\mathcal{B}_{1}^{1}$ or $A_{i}\in\mathcal{B}_{2}^{1}$. Suppose that $A_{i}\in\mathcal{B}_{1}^{1}$ (the argument is similar if $A_{i}\in\mathcal{B}_{2}^{1}$). Pick an $A_{j}\in\mathcal{B}_{2}^{1}$ and an $a'\in A_{j}$. Let $\phi\in\mathcal{G}$ be the transposition  $(a,a')$ (i.e. $\phi$ interchanges $a$ and $a'$ and fixes all other atoms). Then $$\bigcup\{\phi(Z):Z\in\mathcal{B}_{2}^{1}\}=(\bigcup(\mathcal{B}_{2}^{1}\setminus\{A_{j}\}))\cup((A_j\setminus\{a^{\prime}\})\cup\{a\}).$$

Let $E=\{\bigcup\mathcal{B}_{1}^{1},\bigcup\{\phi(Z):Z\in\mathcal{B}_{2}^{1}\}\}$. Then, $E\in[\mathcal{Q}_{1}^{1}\cup\mathcal{Q}_{2}^{1}]^{<\omega}\subset[\mathcal{Q}]^{<\omega}$, so $E\in [\mathcal{Q}]^{<\omega}$. Furthermore, $\mathcal{G}_{E}\subseteq \{\pi\in \mathcal{G}:\pi(a)=a\}$. Indeed, let $\pi\in\mathcal{G}_E$. Towards a contradiction, assume that $\pi(a)=b$ for some $b\in A\setminus\{a\}$. Since $a\in\bigcup\mathcal{B}_{1}^{1}$ and $\pi$ fixes $\bigcup\mathcal{B}_{1}^{1}$, it follows that $b=\pi(a)\in\pi(\bigcup\mathcal{B}_{1}^{1})=\bigcup\mathcal{B}_{1}^{1}$. But then, since $\pi\in\mathcal{G}_{E}$, we have the following:
\begin{multline*}
a\in\bigcup\{\phi(Z):Z\in\mathcal{B}_{2}^{1}\}\}\to \pi(a)\in\pi(\bigcup\{\phi(Z):Z\in\mathcal{B}_{2}^{1}\}\})\\
\to b\in\bigcup\{\phi(Z):Z\in\mathcal{B}_{2}^{1}\}\}=(\bigcup(\mathcal{B}_{2}^{1}\setminus\{A_{j}\}))\cup((A_j\setminus\{a^{\prime}\})\cup\{a\}),
\end{multline*}
which is a contradiction. Therefore, (\ref{eq:normalf1}) holds.
\smallskip

To argue for (\ref{eq:normalf2}), let $\pi\in \mathcal{G}$ and $H\in\mathcal{F}$. There exists $E\in [\mathcal{Q}]^{<\omega}$ such that $\mathcal{G}_{E}\subseteq H$. By the definition of $\mathcal{Q}$, we have $\pi[E]\in[\mathcal{Q}]^{<\omega}$. We assert that $\mathcal{G}_{\pi[E]}\subseteq\pi H\pi^{-1}$. Let $\rho\in \mathcal{G}_{\pi[E]}$. For every $T\in E$ we have the following:
$$\rho(\pi T)=\pi T\rightarrow\pi^{-1}\rho\pi (T)=T;$$
\noindent Hence, since $\mathcal{G}_{E}\subseteq H$, we have:
$$\pi^{-1}\rho\pi\in \mathcal{G}_{E}\rightarrow\rho\in\pi \mathcal{G}_{E}\pi^{-1}\subseteq \pi H\pi^{-1}.$$ 
\noindent Therefore, $\rho\in\pi H\pi^{-1}$. Since $\rho$ is an arbitrary element of $\mathcal{G}_{\pi[E]}$, we conclude that $\mathcal{G}_{\pi[E]}\subseteq\pi H\pi^{-1}$. Thus, $\pi H\pi^{-1}\in\mathcal{F}$, so (\ref{eq:normalf2}) holds. This completes the proof that $\mathcal{F}$ is a normal filter on $\mathcal{G}$.

Let $\mathcal{N}$ be the permutation model determined by $M$, $\mathcal{G}$ and $\mathcal{F}$. We say that an element $x\in\mathcal{N}$ has \emph{support} $E\in [\mathcal{Q}]^{<\omega}$ if, for all $\phi\in \mathcal{G}_{E}$, $\phi(x)=x$. 

In $\mathcal{N}$, the set $(\mathcal{P}(A))^{\mathcal{N}}=\mathcal{P}(A)\cap\mathcal{N}$ is the power set of $A$. To prove that $A$ is dyadically filterbase infinite in $\mathcal{N}$, let  us show that, in $\mathcal{N}$, the discrete space $\langle A, (\mathcal{P}(A))^{\mathcal{N}}\rangle$ satisfies condition ($c$) of Theorem \ref{s1t:main1}. To this aim, for every $n\in\mathbb{N}$ and for every $i\in\{1,2,\ldots,2^{n}\}$, we let
$$\mathcal{V}_{i}^{n}=\left\{\bigcap\mathcal{R}:\mathcal{R}\in [\mathcal{Q}_{i}^{n}]^{<\omega}\setminus\{\emptyset\}\right\}\cup\left\{\bigcup\mathcal{C}: \mathcal{C}\in[\mathcal{Q}_{i}^{n}]^{<\omega}\setminus\{\emptyset\}\right\},$$
and we also let
$$\mathcal{V}=\{\mathcal{V}_{i}^{n}:n\in\mathbb{N},i\in\{1,2,\ldots,2^{n}\}\}.$$

\noindent We notice that any permutation of $A$ in $\mathcal{G}$ fixes $\mathcal{V}$ pointwise. Hence, $\mathcal{V}\in\mathcal{N}$ and, moreover, $\mathcal{V}$ is well-orderable in the model $\mathcal{N}$ (see \cite[page 47]{Je}). Since $\mathcal{V}$ is denumerable in the ground model $M$, it follows that $\mathcal{V}$ is also denumerable in $\mathcal{N}$. Furthermore, in view of the properties of the family $\mathcal{B}$ and of the elements of $\mathcal{G}$, and the construction of $\mathcal{V}$, it is easy to see that, if we put $X=A$ and $\mathbf{X}=\langle A, (\mathcal{P}(A))^{\mathcal{N}}\rangle$, then  $\mathcal{V}$ has properties  (i)-(iv) of condition $c$ of Theorem \ref{s1t:main1}. This, together with Theorem \ref{s1t:main1}, proves that $A$ is dyadically filterbase infinite in the model $\mathcal{N}$.
To complete the proof, it remains to show that $A$ is weakly Dedefind-finite in $\mathcal{N}$.

By way of contradiction, we assume that $A$ is weakly Dedekind-infinite in $\mathcal{N}$. Thus, it holds in $\mathcal{N}$ that there exists a denumerable disjoint family $\mathcal{U}=\{U_{n}:n\in\omega\}$ of $(\mathcal{P}(A))^{\mathcal{N}}$.  Let $E\in [\mathcal{Q}]^{<\omega}$ be a support of $U_{n}$ for all $n\in\omega$. It is not hard to verify now that there exist  a pair $k,m$ of distinct elements of $\omega$ and a pair $x,y$ of atoms, such that $x\in U_{k}$ and $y\in U_{m}$. The transposition $\psi=(x,y)$ from the group $\mathcal{G}$ is an element of $\mathcal{G}_{E}$. It follows that $\psi(U_{k})=U_{k}$, and so $y=\psi(x)\in\psi(U_{k})=U_{k}$. This is impossible because $y\in U_{m}$ and $U_{k}\cap U_{m}=\emptyset$. The contradiction obtained shows that $A$ is weakly Dedekind-finite in $\mathcal{N}$.
\end{proof}

The model constructed in \cite[the proof to Theorem 5.14]{ktw1} is different from that we have just introduced in the proof to Theorem \ref{thm:3}. By Theorem 5.15 of \cite{ktw1}, $\mathbf{NAS}$ is false in the model from \cite[the proof to Theorem 5.14]{ktw1}. 

Let $\mathcal{N}$ be the model from the proof to Theorem \ref{thm:3} above. We do not know if $\mathbf{NAS}$ holds  $\mathcal{N}$. Let us notice that, for every $i\in\omega$, no $E \in [\mathcal{Q}]^{<\omega}$ can support $A_i$ and, in consequence, $A_i\notin\mathcal{N}$. This is why we cannot mimic the proof to Theorem 5.15 in \cite{ktw1} to show that $\mathbf{NAS}$ fails in $\mathcal{N}$.  We also do not know if $\mathbf{INSHC}$ holds in $\mathcal{N}$.

\section{A metrization theorem for a class of quasi-metrizable spaces}
\label{s3}

The following theorem is of significant importance because of its consequences that will be shown in the forthcoming results.

\begin{theorem}
\label{s2t4}
$(\mathbf{ZF})$
Let $d$ be a quasi-metric on a set $X$ such that either $d^{-1}$ is precompact or the space $\mathbf{X}=\langle X, \tau(d)\rangle$ is limit point compact.  Then $\Iso_{\tau(d^{-1})}(X)$ is a cuf set. 
\end{theorem}
\begin{proof}
Let $D=\Iso_{\tau(d^{-1})}(X)$. For every $x\in D$, let
$$n_x=\min\left\{n\in\mathbb{N}: B_{d^{-1}}\left(x, \frac{1}{n}\right)=\{x\}\right\}.$$
For every $n\in\mathbb{N}$, let $A_n=\{x\in D: n=n_x\}$. Suppose that $n_0\in\mathbb{N}$ is such that $A_{n_0}$ is infinite. Let us show that there exist $x_0\in X$ and $x_1\in A_{n_0}$ such that $x_0\neq x_1$ and $x_1\in B_d\left(x_0,\frac{1}{n_0}\right)$. If $x_0, x_1$ are such  points, we notice that $d^{-1}(x_1, x_0)=d(x_0, x_1)<\frac{1}{n_0}$, so $x_0\in B_{d^{-1}}\left(x_1,\frac{1}{n_0}\right)$ which is impossible by the definition of $A_{n_0}$. 

If $\mathbf{X}$ is limit point compact, there exists an accumulation point of $A_{n_0}$ in $\mathbf{X}$. In this case, for a fixed accumulation point $x_0$ of $A_{n_0}$, we can fix $x_1\in A_{n_0}$ such that $x_1\neq x_0$ and $x_1\in B_d\left(x_0, \frac{1}{n_0}\right)$.  Assuming that $d^{-1}$ is precompact, we can fix a finite set $F\subseteq X$ such that $X=\bigcup\limits_{x\in F}B_d\left(x, \frac{1}{n_0}\right)$ and, since $A_{n_0}$ is infinite, we can fix $x_0\in F$ and $x_1\in A_{n_0}$ such that $x_1\neq x_0$ and $x_1\in B_d\left(x_0, \frac{1}{n_0}\right)$.  Hence, the assumption that $A_{n_0}$ is infinite leads to a contradiction. Therefore,  $D=\bigcup\limits_{n\in\mathbb{N}}A_n$ is a cuf set.
\end{proof}

\begin{corollary} 
\label{s2c5}
$(\mathbf{ZF})$ Let $\mathbf{X}=\langle X, d\rangle$ be a metric space which is either limit point compact or totally bounded. Then $\Iso(X)$ is a cuf set. Furthermore, if $\Iso(X)$ is infinite, then $X$ is quasi Dedekind-infinite.
\end{corollary}

\begin{proof} That $\Iso(X)$ is a  cuf set follows from Theorem \ref{s2t4}. The second assertion is straightforward. 
\end{proof}

\begin{remark}
\label{s2r6}
Let $\mathbf{X}$ be a compact metrizable space. Then, using Proposition \ref{discommet}(ii), we may deduce that $\Iso(X)$ is a cuf set. Namely, suppose that $\Iso(X)$ is infinite. Let $Y=\text{cl}_{\mathbf{X}}(\Iso(X))$. Then $\mathbf{Y}$ is a metrizable compactification of the discrete space $\Iso(X)$, so $\Iso(X)$ is a cuf set by Proposition \ref{discommet}(ii). 
\end{remark}

Theorem \ref{qm1}($b$) improves the well-known result of $\mathbf{ZFC}$ that every compact Hausdorff quasi-metrizable space is metrizable (see Corollary  in \cite[Corollary in 7.1, p. 153]{fl} since it establishes that the weaker (than $\mathbf{AC}$) choice principle $\mathbf{CAC}$ suffices for the proof. An open problem posed in \cite{kw1} is whether it can be proved in $\mathbf{ZF}$ that every quasi-metrizable compact Hausdorff space is metrizable. Theorem \ref{qm1} is a partial solution to this problem. Now, we can shed a little more light on it by the following theorem:

\begin{theorem}
\label{s2t7}
$(\mathbf{ZF})$ Let $d$ be a strong quasi-metric on a set $X$ such that $\langle X, \tau(d)\rangle$ is a $T_3$-space. Then the following conditions are satisfied:
\enumerate
\item[(i)] if $\langle X, \tau(d^{-1})\rangle$ has a dense cuf set, then $\langle X, \tau(d)\rangle$ is metrizable;
\item[(ii)] if  $\langle X, \tau( d^{-1})\rangle$ is iso-dense and either $\langle X, \tau(d)\rangle$ is limit point compact or $d^{-1}$ is precompact, then the space $\langle X, \tau(d)\rangle$ is metrizable.
\end{theorem}
\begin{proof}
(i) Assume that $A=\bigcup\limits_{n\in\omega}A_n$ is a dense set in $\langle X, \tau(d^{-1})\rangle$ such that, for every $n\in\omega$, $A_n$ is a non-empty finite set. For $n,m\in\omega$, we define $\mathcal{B}_{m,n}=\left\{B_d\left(x,\frac{1}{m+1}\right): x\in A_n\right\}$. Since $d$ is strong, in much the same way, as in the proof to Theorem 4.6 in \cite{kw1}, one can show that  $\mathcal{B}=\bigcup\limits_{n,m\in\omega}\mathcal{B}_{n,m}$ is a base of $\langle X, \tau(d)\rangle$. Since $\mathcal{B}$ is a cuf set, the space $\langle X, \tau(d)\rangle$ is metrizable by Theorem \ref{Umt}(ii).\smallskip

(ii) Now, we assume that  $\langle X, \tau( d^{-1})\rangle$ is iso-dense and either $\langle X, \tau(d)\rangle$ is limit point compact or $d^{-1}$ is precompact. Let $E=\Iso_{\tau(d^{-1})}(X)$. Then $E$ is dense in $\langle X, \tau(d^{-1})\rangle$.  By Theorem \ref{s2t4}, the set $E$ is a cuf set. Hence, to conclude the proof, it suffices to apply (i). 
\end{proof}

\begin{corollary} 
\label{s2c8}
$(\mathbf{ZF})$
Let $\langle X, d\rangle$ be a compact Hausdorff quasi-metric space such that $\langle X, \tau(d^{-1})\rangle$ is iso-dense. Then $\langle X, \tau(d)\rangle$ is metrizable.
\end{corollary}
\begin{proof}
This follows immediately from Proposition \ref{s1p12} and Theorem \ref{s2t7}.
\end{proof}

\begin{remark}
\label{s2r9}
In Corollary \ref{s2c8}, we cannot omit the assumption that $\langle X, \tau(d)\rangle$ is Hausdorff. Indeed, there is a quasi-metric $d$ on $\omega$ such that $\tau(d)$ is the cofinite topology on $\omega$ and $\tau(d^{-1})$ is the discrete topology on $\omega$ (see \cite{kw1}). Then $d$ is a strong quasi-metric such that $\langle \omega, \tau(d)\rangle$ is a compact $T_1$-space which is not metrizable. 
\end{remark}

\begin{theorem}
\label{s2t10}
$(\mathbf{ZF})$
Let $\langle X, d\rangle$ be an iso-dense metric space such that either $d$ is totally bounded or $\langle X, \tau(d)\rangle$ is limit point compact. Then $\langle X, \tau(d)\rangle$ has a cuf base and can be embedded in a metrizable Tychonoff cube.
\end{theorem}
\begin{proof}
It follows from the proof to Theorem \ref{s2t7} that $\langle X, \tau(d)\rangle$ has a cuf base. Since $\langle X, \tau(d)\rangle$ is a $T_3$-space, to conclude the proof, it suffices to apply Theorem \ref{Umt}(ii).
\end{proof}

\section{$\mathbf{CAC}_{fin}$ via iso-dense metrizable spaces}
\label{s4}

It is known that it holds in $\mathbf{ZFC}$ that every iso-dense compact metrizable space is separable and every scattered compact metrizable space is countable. In this section, we show that the situation of compact iso-dense metrizable spaces and compact scattered metrizable spaces in $\mathbf{ZF}$ is different than in $\mathbf{ZFC}$. To begin, let us recall the following lemma proved in \cite{ktw0}:

\begin{lemma}
\label{s3l01}
$(\mathbf{ZF})$. Let $\mathbf{X}$ be a non-empty metrizable space and let $\mathcal{B}$ be a base of $\mathbf{X}$. Then $\mathbf{X}$ embeds in $[0, 1]^{\mathcal{B}\times\mathcal{B}}$. 
\end{lemma}

If $\mathbf{X}=\langle X, d\rangle$ is a metric space and $\mathbf{Y}$ is a topological space, then we say that $\mathbf{X}$ embeds in  $\mathbf{Y}$ if the space $\langle X, \tau(d)\rangle$ embeds in $\mathbf{Y}$.

The following theorem is a characterization of $\mathbf{CAC}_{fin}$ in terms of iso-dense (limit point) compact metrizable spaces and in terms of iso-dense totally bounded metric spaces.

\begin{theorem}
\label{s3t02}
 $(\mathbf{ZF})$
The following conditions are all equivalent:
\begin{enumerate}
\item[(i)] $\mathbf{CAC}_{fin}$;
\item[(ii)] for every iso-dense metric space $\mathbf{X}$, if $\mathbf{X}$ is either limit point compact or totally bounded, then $\mathbf{X}$ is separable;
\item[(iii)] for every iso-dense metric space $\mathbf{X}$, if $\mathbf{X}$ is either limit point compact or totally bounded, then $\mathbf{X}$ embeds in the Hilbert cube $[0,1]^{\mathbb{N}}$;
\item[(iv)] for every iso-dense metric space $\mathbf{X}$, if $\mathbf{X}$ is either limit point compact or totally bounded, then $|\Iso(X)|\leq|\mathbb{R}|$;
\item[(v)] for every iso-dense metric space $\mathbf{X}$, if $\mathbf{X}$ is either limit point compact or totally bounded, then the set $\Iso(X)$ is countable.
\end{enumerate}
In (ii)-(v), the term ``iso-dense''  can be replaced with ``scattered''.
\end{theorem} 
\begin{proof}
Let $\mathbf{X}=\langle X, d\rangle$ be an iso-dense (respectively, scattered) metric space such that $\mathbf{X}$ is either limit point compact or totally bounded.  By Corollary \ref{s2c5}, the set $\Iso(X)$ is a cuf set. Hence, it follows from $\mathbf{CAC}_{fin}$ that $\Iso(X)$ is countable. In consequence, (i) implies (ii). Since every separable metrizable space is second-countable, it follows from Lemma  \ref{s3l01} that it is true in $\mathbf{ZF}$ that (ii) implies (iii). 

Now, to show that (iii) implies (iv),  suppose that $\langle X, \tau(d)\rangle$ is homeomorphic to a subspace of $[0, 1]^{\mathbb{N}}$. Then $\Iso(X)$ is equipotent to a subset of $[0, 1]^{\mathbb{N}}$. Since it holds in $\mathbf{ZF}$ that $[0, 1]^{\mathbb{N}}$ and $\mathbb{R}$ are equipotent, we deduce that $\Iso(X)$ is equipotent to a subset of $\mathbb{R}$. Hence (iii) implies (iv). 

It is obvious that, in $\mathbf{ZF}$, every cuf subset of $\mathbb{R}$ is countable as a countable union of finite well-ordered sets. Hence, if $\Iso(X)$ is equipotent to a subset of $\mathbb{R}$, then $\Iso(X)$ is countable as a set equipotent to a cuf set contained in $\mathbb{R}$. This shows that (iv) implies (v). 

Finally, suppose that $\mathbf{CAC}_{fin}$ fails. Then there exists an uncountable discrete cuf space $\mathbf{D}$. It follows from Proposition \ref{discommet}(i) that the Alexandroff compactification $\mathbf{D}(\infty)$ of $\mathbf{D}$ is metrizable. Since $\mathbf{D}(\infty)$ is an iso-dense compact mertizable space whose set of all isolated points $\mathbf{D}(\infty)$ is uncountable, (v) fails if $\mathbf{CAC}_{fin}$ fails. Hence (v) implies (i).
\end{proof}

\begin{theorem}
\label{s3t04}
$(\mathbf{ZF})$ 
\begin{enumerate}
\item[(i)]For every totally bounded metric space $\mathbf{X}$, the Cantor-Bendixson rank $|\mathbf{X}|_{\CB}$ of  $\mathbf{X}$ is a countable ordinal.
\item[(ii)] Every totally bounded scattered metric space is a cuf space.
\item[(iii)] Every totally bounded scattered metric space has a cuf base. 
\end{enumerate}
\end{theorem}
\begin{proof}
Let $\mathbf{X}=\langle X, d\rangle$ be an infinite totally bounded  metric space. Let $\alpha=|\mathbf{X}|_{\text{CB}}$. Then 
$$X=\bigcup\limits_{\gamma\in\alpha}\Iso(X^{(\gamma)})\cup X^{(\alpha)}.$$
For every $\gamma\in\alpha$ and every $x\in\Iso(X^{(\gamma)})$, let 
$$n(x,\gamma)=\min\left\{n\in\mathbb{N}: B_d\left(x,\frac{1}{n}\right)\cap X^{(\gamma)}=\{x\}\right\}.$$
 For every $\gamma\in\alpha$ and every $n\in\mathbb{N}$, let 
$$A_{\gamma, n}=\{x\in \Iso(X^{(\gamma)}): n(x,\gamma)=n\}.$$
We have already shown in the proof to Theorem \ref{s2t4} that, for every $\gamma\in\alpha$ and every $n\in\mathbb{N}$, the set $A_{\gamma, n}$ is finite and $\Iso(X^{(\gamma)})=\bigcup\limits_{n\in\mathbb{N}}A_{\gamma, n}$.\smallskip

(i) Suppose that $\alpha$  is uncountable. For every $n\in\mathbb{N}$, let 
$$B_n=\{\gamma\in\alpha: A_{\gamma, n}\neq\emptyset\}.$$

\noindent Since $\alpha$ is supposed to be uncountable, there exists $n_0\in\mathbb{N}$ such that $B_{n_0}$ is infinite. We fix such an $n_0$ and put
$$\mathcal{U}=\left\{B_d\left(x, \frac{1}{3n_0}\right): x\in X\right\}.$$
By the total boundedness of $d$, the open cover $\mathcal{U}$ of $\mathbf{X}$ has a finite subcover. Hence, there exists a non-empty finite subset $F$ of $X$ such that $X=\bigcup\limits_{x\in F}B_d\left(x,\frac{1}{3n_0}\right)$. Since $B_{n_0}$ is infinite, there exist $\gamma_1,\gamma_2\in B_{n_0}$ and elements $x_0\in F$,  $x_1\in A_{\gamma_1, n_0}\cap B_d\left(x_0,\frac{1}{3n_0}\right)$ and $x_2\in A_{\gamma_2, n_0}\cap B_d\left(x_0, \frac{1}{3n_0}\right)$, such that $x_1\neq x_2$. We may assume that $\gamma_1\leq\gamma_2$. Then $X^{(\gamma_2)}\subseteq X^{(\gamma_1)}$ and it follows from the definition of $A_{\gamma_1, n_0}$ that $d(x_1, x_2)\geq\frac{1}{n_0}$. On the other hand, since $x_1, x_2\in B_d(x_0,\frac{1}{3n_0})$, we have $d(x_1, x_2)\leq\frac{2}{3n_0}<\frac{1}{n_0}$. The contradiction obtained proves that  $\alpha$ is countable.\smallskip

(ii) Now, suppose that the space $\mathbf{X}$ is also scattered. Then it follows from  Proposition \ref{s1p10} that $X^{(\alpha)}=\emptyset$. Hence $X=\bigcup\{A_{\gamma, n}: \gamma\in\alpha\text{ and } n\in\mathbb{N}\}$. Since $\alpha$ is countable, the set $\alpha\times\mathbb{N}$ is countable. This implies that the family $\{A_{\gamma, n}: \gamma\in\alpha\text{ and }n\in\mathbb{N}\}$ is also countable. We have already shown that, for every $\gamma\in\alpha$ and every $n\in\mathbb{N}$, the set $A_{\gamma, n}$ is finite. Hence $X$ is a cuf set.

It follows immediately from (ii) and Theorem \ref{s2t10} that (iii) holds.
\end{proof}

\begin{theorem}
\label{s3c05}
$(\mathbf{ZF})$ The following conditions are all equivalent:
\begin{enumerate}
\item[(i)] $\mathbf{CAC}_{fin}$;
 \item[(ii)] every totally bounded scattered metric space is countable;
 \item[(iii)] every compact metrizable scattered space is countable;
 \item[(iv)] every totally bounded, complete scattered metric space is compact.
 \end{enumerate}
\end{theorem}
\begin{proof}
Since $\mathbf{CAC}_{fin}$ implies that all cuf sets are countable, it follows from Theorem \ref{s3t04} that (i) implies (ii) and (iii). It is provable in $\mathbf{ZF}$ that every totally bounded, complete countable metric space is compact. Hence, in the light of Theorem \ref{s3t04}, (i) implies (iv).  

Assume that $\mathbf{CAC}_{fin}$ is false. Then there exists a family $\{A_n: n\in\omega\}$ of non-empty pairwise disjoint finite sets such that the set $D=\bigcup\limits_{n\in\omega}A_n$ is Dedekind-finite (see \cite[ Form \textbf{[10M]}]{hr}). Let $\mathbf{D}=\langle D, \mathcal{P}(D)\rangle$. By Proposition \ref{discommet}(i), the space $\mathbf{D}(\infty)$ is metrizable. Let $d$ be any metric which induces the topology of $\mathbf{D}(\infty)$. Since $\mathbf{D}(\infty)$ is compact, the metric $d$ is totally bounded.  Moreover, $\mathbf{D}(\infty)$ is scattered but uncountable. For $\rho=d\upharpoonright D\times D$, the metric space $\langle D, \rho\rangle$ is also totally bounded. Since $D$ is Dedekind-finite and $\langle D, \rho\rangle$ is discrete, the metric $\rho$ is complete. Clearly, $\langle D, \rho\rangle$ is not compact. All this taken together completes the proof.  
\end{proof}

\begin{theorem}
\label{s3p06}
$(\mathbf{ZF})$ Every compact metrizable cuf space is scattered. In particular, every compact metrizable countable space is scattered.
\end{theorem}
\begin{proof}
Our first step is to prove that every non-empty compact metrizable cuf space has an isolated point. To this aim, suppose that $\mathbf{X}=\langle X, d\rangle$ is a compact metric space such that the set $X$ is a non-empty cuf set. Towards a contradiction, suppose that $\mathbf{X}$ is dense-in-itself. We fix a partition $\{X_n: n\in\omega\}$ of $\mathbf{X}$ into non-empty finite sets.

Let $S=\bigcup\{\{0,1\}^n: n\in\mathbb{N}\}$. For $n\in\mathbb{N}$, $s\in\{0,1\}^n$ and $t\in\{0,1\}$, let $s\smallfrown t\in\{0,1\}^{n+1}$ be defined by: $s\smallfrown t(i)=s(i)$ for every $i\in n$, and $s\smallfrown t(n)=t$. Using ideas from \cite{CP}, let us define by induction (with respect to $n$) a family $\{B_s: s\in S\}$ such that, for every $s\in S$, the following conditions are satisfied:
\begin{enumerate}
\item[(1)]  $B_s$ is a non-empty open subset of $\mathbf{X}$;
\item[(2)] for every $t\in\{0,1\}$, $B_{s\smallfrown t}\subseteq B_s$;
\item[(3)]  $\cl_{\mathbf{X}}(B_{s\smallfrown 0})\cap\cl_{\mathbf{X}}(B_{s\smallfrown 1})=\emptyset$.
\end{enumerate}

To start the induction, for $n=1=\{0\}$ and  every $s\in\{0,1\}^1$, we define:
$$B_s=\bigcup\left\{ B_d\left(x, \frac{d(X_0, X_1)}{3}\right): x\in X_{s(0)}\right\}.$$
\noindent Now, suppose that $n\in\mathbb{N}$ is such that, for every $s\in\bigcup\limits_{i=1}^n\{0,1\}^i$, we have defined a non-empty open subset $B_s$ of $\mathbf{X}$.  For an arbitrary $s\in\{0,1\}^{n+1}$, we consider the set $B_{s\upharpoonright n}$.  We put $n_s=\min\{m\in\omega: X_m\cap B_{s\upharpoonright n}\neq\emptyset\}$. Since $\mathbf{X}$ is  dense-in-itself, we have $\emptyset\neq\{m\in\omega: X_m\cap (B_{s\upharpoonright n}\setminus X_{n_s})\neq\emptyset\}$, so we can define $k_s=\min\{m\in\omega: X_m\cap (B_{s\upharpoonright n}\setminus X_{n_s})\neq\emptyset\}$. Now, we put $Y_{s,0}=X_{n_s}\cap B_{s\upharpoonright n}$ and $Y_{s,1}=X_{k_s}\cap (B_{s\upharpoonright n}\setminus X_{n_s})$. We define
$$B_s=\bigcup\left\{B_d\left(y,\frac{d(Y_{s,0}, Y_{s,1})}{3}\right): y\in Y_{s,s(n)}\right\}.$$ 
In this way, we have inductively defined the required family $\{B_s: s\in S\}$. 

We notice that it follows from (2) that, for every $f\in \{0,1\}^{\omega}$ and $n\in\mathbb{N}$,  $\emptyset\neq \cl_{\mathbf{X}}(B_{f\upharpoonright (n+1)})\subseteq \cl_{\mathbf{X}}(B_{f\upharpoonright n})$. Thus, by the compactness of $\mathbf{X}$, for every $f\in\{0,1\}^{\omega}$, the set 
$$M_f=\bigcap\left\{\cl_{\mathbf{X}}(B_{f\upharpoonright n}): n\in\mathbb{N}\right\}$$
is non-empty. For every $f\in\{0, 1\}^{\omega}$, let $m_f=\min\{n\in\omega: X_n\cap M_f\neq\emptyset\}$. We define a mapping $F:\{0, 1\}^{\omega}\to \bigcup\{\mathcal{P}(X_n): n\in\omega\}$ by putting:
 $$F(f)=X_{m_f}\cap M_f\text{ for every } f\in\{0,1\}^{\omega}.$$

 It follows from (3) that $F$ is an injection. In consequence, the set $\{0, 1\}^{\omega}$ is equipotent to a subset of the cuf set $\bigcup\{\mathcal{P}(X_n): n\in\omega\}$. But this is impossible because $\{0, 1\}^{\omega}$, being equipotent to $\mathbb{R}$, is not a cuf set. The contradiction obtained shows that every non-empty compact metrizable cuf space has an isolated point.

To complete the proof, we let $\mathbf{X}$ be any compact metrizable cuf space. We have proved that every non-empty compact subspace of $\mathbf{X}$ has an isolated point. Hence $\mathbf{X}$ cannot contain non-empty dense-in-itself subspaces. This implies that $\mathbf{X}$ is scattered.  
\end{proof}

Now, we can give the following modification of  Theorem \ref{s1t11}:

\begin{theorem}
\label{s4t06}
$(\mathbf{ZF})$ Let $\mathbf{X}$ be a compact Hausdorff, non-scattered space which has a cuf base. If $\mathbf{X}$ is weakly Loeb, then $|\mathbb{R}|\leq|[X]^{<\omega}|$. If $\mathbf{X}$ is Loeb, then $|\mathbb{R}|\leq|X|$. 
\end{theorem}
\begin{proof}
Without loss of generality, we may assume that $\mathbf{X}$ is dense-in-itself because  we can replace $\mathbf{X}$ with its non-empty dense-in-itself compact subspace. By Theorem 1.13(ii), $\mathbf{X}$ is metrizable. It is known that every compact metrizable Loeb space is second-countable (see, e.g., \cite{KT2}). Hence, if $\mathbf{X}$ is Loeb, then $|\mathbb{R}|\leq |X|$ by Theorem \ref{s1t11}. Suppose that $\mathbf{X}$ is weakly Loeb.  Let $d$ be any metric on $X$ which induces the topology of $\mathbf{X}$. It follows from Proposition \ref{s1p15} that $\mathbf{X}$ has a dense cuf set. Since $\mathbf{X}$ is non-empty and dense-in-itself, every dense subset of $\mathbf{X}$ is infinite. Therefore, we can fix a disjoint family $\{X_n: n\in\omega\}$ of non-empty finite subsets of $\mathbf{X}$ such that the set $\bigcup\{X_n: n\in\omega\}$ is dense in $\mathbf{X}$. Mimicking the proof to Theorem 4.5, we can define an injection $F:\{0,1\}^{\omega}\to\mathcal{P}(X)$ such that, for every $f\in\{0,1\}^{\omega}$, the set $M_f=F(f)$ is a non-empty closed subset of $\mathbf{X}$ and, for every pair $f,g$ of distinct functions from $\{0,1\}^{\omega}$, $M_f\cap M_g=\emptyset$. Let $\psi$ be a weak Loeb function for $\mathbf{X}$. Then $\psi\circ F$ is an injection from $\{0, 1\}^{\omega}$ into $[X]^{<\omega}$. Hence $|\mathbb{R}|=|\{0, 1\}^{\omega}|\leq|[X]^{<\omega}|$. 
\end{proof}

Taking the opportunity,  let us give a proof to the following theorem:

\begin{theorem}
\label{woac}
$(\mathbf{ZF})$
\begin{enumerate}
\item[($a$)] For every non-empty set $I$ and every family $\{\langle X_i, \tau_i\rangle: i\in I\}$ of denumerable metrizable compact spaces, the family $\{X_i: i\in I\}$ has a multiple choice function.
\item[($b$)] For a non-zero von Neumann ordinal $\alpha$, let $\{\mathbf{X}_\gamma: \gamma\in\alpha\}$ be a family of pairwise disjoint non-empty  countable, compact metrizable spaces. Then the direct sum $\mathbf{X}=\bigoplus\limits_{\gamma\in\alpha}\mathbf{X}_{\gamma}$ is weakly Loeb.
\item[($c$)] The following conditions are all equivalent:
\begin{enumerate}
\item[(i)] $\mathbf{WOAC}_{fin}$;
\item[(ii)] for every well-orderable set $S$ and every family $\{\langle X_s, d_s\rangle: s\in S\}$ of scattered totally bounded metric spaces, the union $\bigcup\limits_{s\in S}X_s$ is well-orderable. 

\item[(iii)] for every well-orderable non-empty set $S$ and every family $\{\langle X_s, d_s\rangle: s\in S\}$ of compact scattered metric spaces, the product $\prod\limits_{s\in S}\langle X_s, \tau(d_s)\rangle$ is compact. 
\end{enumerate}

\end{enumerate}
\end{theorem}

\begin{proof}
($a$) Let $\mathbf{X}_i=\langle X_i, \tau_i\rangle$ be a denumerable metrizable compact space for every $i\in I$ with $I\neq\emptyset$. For every $i\in I$, let $\alpha_i=|\mathbf{X}_i|_{\CB}$. We fix $i\in I$ and observe that,  since  $\mathbf{X}_i$ is scattered by Theorem \ref{s3p06}, it follows from Proposition \ref{s1p10} that $X_i^{(\alpha_i)}=\emptyset$ and $\alpha_i$ is a successor ordinal. Let $\beta_i\in ON$ be such that $\alpha_i=\beta_i+1$. Then $X_i^{(\beta_i)}=\Iso(X_i^{(\beta_i)})$. If the set $\Iso(X_i^{(\beta_i)})$ were infinite, it would have an accumulation point in $X_i^{(\beta_i)}$ by the compactness of $\mathbf{X}_i$. Hence $\Iso(X_i^{(\beta_i)})$ is a non-empty finite set. In consequence, by assigning to any $i\in I$ the set $\Iso(X_i^{(\beta_i)})$, we obtain a multiple choice function of $\{X_i: i\in I\}$.

($b$) Let us consider the family $\mathcal{F}$ of all non-empty closed sets of $\mathbf{X}$.  By the proof of ($a$), there exists a family $\{f_{\gamma}: \gamma\in\alpha\}$ such that, for every $\gamma\in\alpha$, $f_{\gamma}$ is a weak Loeb function of $\mathbf{X}_{\gamma}$. For every $F\in\mathcal{F}$, let $\gamma(F)=\min\{\gamma\in\alpha: F\cap X_{\gamma}\neq\emptyset\}$ and let $f(F)=f_{\gamma(F)}(F\cap X_{\gamma})$. Then $f$ is a weak Loeb function of  $\mathbf{X}$.

($c$) (i)$\rightarrow$(ii) Let us assume $\mathbf{WOAC}_{fin}$. Suppose that  $S$ is a well-orderable non-empty set and, for every $s\in S$, $\mathbf{X}_s=\langle X_s, d_s\rangle$ is a non-empty scattered  totally bounded metric space. To prove that $X=\bigcup\limits_{s\in S}X_s$ is well-orderable,  without loss of generality, we may assume that $S=\alpha$ for some non-zero von Neumann ordinal $\alpha$, and $X_i\cap X_j=\emptyset$ for every pair $i,j$ of distinct elements of $\alpha$. In much the same way, as in the proof to Theorem \ref{s3t04}, we can define a family $\{A_{i, n}: i\in\alpha, n\in\mathbb{N}\}$ of non-empty finite sets such that, for every $i\in\alpha$, $X_i=\bigcup\limits_{n\in\mathbb{N}}A_{i,n}$. Now, we can easily define a family $\{M_i: i\in\alpha\}$ of subsets of $\mathbb{N}$ and a family $\{F_{i,n}: i\in\alpha, n\in M_i\}$ of pairwise disjoint non-empty finite sets such that, for every $i\in\alpha$, $X_i=\bigcup\limits_{n\in M_i}F_{i,n}$. The set $J=\{\langle i, n\rangle: i\in\alpha, n\in M_i\}$ is well-orderable, so we can fix a von Neumann ordinal number $\gamma$ and a bijection $h:\gamma\to J$. For every $j\in\gamma$, let $n(j)\in\omega$ be equipotent to $F_{h(j)}$, and let $B_j=\{ f\in F_{h(j)}^{n(j)}: f\text{ is a bijection}\}$. By $\mathbf{WOAC}_{fin}$, there exists $\psi\in\prod\limits_{j\in\gamma}B_{j}$. Now, we can define a well-ordering $\leq$ on $X=\bigcup\limits_{j\in\gamma}F_{h(j)}$ as follows: for  $i,j\in\gamma$, $x\in F_{h(i)}, y\in F_{h(j)}$, we put $x\leq y$ if either $i\in j$ or $i=j$ and $\psi(i)^{-1}(x)\subseteq \psi(i)^{-1}(y)$. Hence (i) implies (ii)

(ii)$\rightarrow$(iii) Let $S$ be a non-empty well-orderable set and, for every $s\in S$, let $\langle X_s, d_s\rangle$ be a compact scattered metric space. Clearly, we may assume that $S$ is a von Neumann cardinal. We notice that if $X=\bigcup\limits_{s\in S}X_s$ is well-orderable, then we can define a family $\{f_s: s\in S\}$ such that, for every $s\in S$, $f_s$ is a Loeb function of $\mathbf{X}_s=\langle X_s, \tau(d_s)\rangle$ and, therefore, by Theorem \ref{t:loeb}, the product $\prod\limits_{s\in S}\mathbf{X}_s$ is compact. Hence (ii) implies (iii).

(iv)$\rightarrow$(i) Now, let $S$ be a well-orderable set and let $\{A_s: s\in S\}$ be a family of non-empty finite sets. Take an element $\infty\notin\bigcup\limits_{s\in S}A_s$ and put $Y_s=A_s\cup\{\infty\}$ for $s\in S$. Let $\rho_s$ be the discrete metric on $Y_s$ and let $\mathbf{Y}_s=\langle Y_s, \mathcal{P}(Y_s)\rangle$ for every $s\in S$. Assuming (iii), we obtain that the space $\mathbf{Y}=\prod\limits_{s\in S}\mathbf{Y}_s$ is compact. In much the same way, as in the proof of Kelley's theorem that the Tychonoff Product Theorem implies $\mathbf{AC}$ (see \cite{kl}), one can show that the compactnes of $\mathbf{Y}$ implies that $\prod\limits_{s\in S}A_s\neq\emptyset$. Hence (iii) implies (i).
\end{proof}

\section{Alexadroff compactifications of infinite discrete cuf spaces as subspaces}
\label{s5}

In this section, we shed some light on the problem of whether it is provable in $\mathbf{ZF}$ that every non-empty dense-in-itself compact metrizable space contains an infinite compact scattered subspace. Unfortunately, we are still unable to give a satisfactiory solution to this problem. Certainly, $\mathbb{N}(\infty)$ or, more generally, for every infinite discrete cuf space $\mathbf{D}$, $\mathbf{D}(\infty)$ is a simple example of an infinite compact metrizable scattered space. So, let us search for conditions on a non-discrete compact metrizable space to contain a copy of $\mathbf{D}(\infty)$ for some infinite discrete cuf space $\mathbf{D}$. Let us start with the following simple proposition:

\begin{proposition}
\label{s4p01}
$(\mathbf{ZF})$
Let $\mathbf{X}$ be a non-discrete first-countable Loeb $T_3$-space. Then $\mathbf{X}$ contains a copy of $\mathbb{N}(\infty)$. In particular, every non-discrete metrizable Loeb space contains a copy of $\mathbb{N}(\infty)$.
\end{proposition}
\begin{proof}
Let $x_0$ be an accumulation point of $\mathbf{X}$ and let $f$ be a Loeb function of $\mathbf{X}$. Since $\mathbf{X}$ is a first-countable $T_3$-space and $x_0$ is an accumulation point of $\mathbf{X}$, there exists a base $\{U_n: n\in\mathbb{N}\}$ of open neighborhoods of $x_0$ such that $\text{cl}_{\mathbf{X}}(U_{n+1})\subset U_n$ for every $n\in\mathbb{N}$. Let $x_n=f(\text{cl}_{\mathbf{X}}(U_n)\setminus U_{n+1})$ for every $n\in\mathbb{N}$. Then the subspace $\{x_{2n}: n\in\omega\}$ of $\mathbf{X}$ is a copy of $\mathbb{N}(\infty)$.
\end{proof}

\begin{theorem} 
\label{s4t2}
$(\mathbf{ZF})$ 
\begin{enumerate}
\item[(i)] $\mathbf{IQDI}$ implies that every infinite compact first-countable Hausdorff space contains a copy of $\mathbf{D}(\infty)$ for some infinite discrete cuf space $\mathbf{D}$. 
\item[(ii)] Each of $\mathbf{IDI}$, $\mathbf{WoAm}$ and $\mathbf{BPI}$  implies that every infinite compact first-countable Hausdorff space contains a copy of $\mathbb{N}(\infty)$. 
\item[(iii)] Every infinite first-countable  compact Hausdorff separable space contains a copy of $\mathbb{N}(\infty)$. Every infinite first-countable compact Hausdorff space having an infinite cuf subset contains a copy of $\mathbf{D}(\infty)$ for some infinite discrete cuf space. 
\end{enumerate}
\end{theorem}
\begin{proof}
Let $\mathbf{X}=\langle X, \tau\rangle$ be an infinite compact first-countable Hausdorff space. For a point $x_0\in X$, let  $\{U_n: n\in\omega\}$ be a base of neighborhoods of $x_0$ such that $\text{cl}_{\mathbf{X}}(U_{n+1})\subseteq U_n$ for every $n\in\omega$.\smallskip

(i) Assuming $\mathbf{IQDI}$, we can fix a disjoint family $\{F_n: n\in\omega\}$ of non-empty finite subsets of $X$. Since $\mathbf{X}$ is compact, the set $F=\bigcup\limits_{n\in\omega}F_n$ has an accumulation point. Let $x_0$ be an accumulation point of $F$. We may assume that $x_0\notin F$. Let $n_0=m_0=\min\{n\in\omega: U_0\cap F_{n}\neq\emptyset\}$, $M_0=\omega$ and $Y_0=F_{n_0}\cap U_0$. Since $\mathbf{X}$ is a $T_1$-space, $Y_0$ is a finite set and $x_0\notin Y_0$, the set $M_1=\{n\in\omega: U_n\cap Y_0=\emptyset\}$ is non-empty. Let $m_1=\min M_1$, $n_1=\min\{n\in\omega: U_{m_1}\cap F_n\neq\emptyset\}$ and $Y_1= U_{m_1}\cap F_{n_1}$.  Suppose that $k\in\omega\setminus\{0\}$ is such that we have already defined $n_k, m_k\in\omega$ such that $Y_k=U_{m_k}\cap F_{n_k}\neq\emptyset$. We put $M_{k+1}=\{n\in\omega: U_n\cap Y_{k}=\emptyset\}$, $m_{k+1}=\min M_{k+1}$, $n_{k+1}=\min\{n\in\omega: U_{m_{k+1}}\cap F_n\neq\emptyset\}$ and $Y_{k+1}=U_{m_{k+1}}\cap F_{n_{k+1}}$. This terminates our inductive definition. Let $D=\bigcup\limits_{k\in\omega}Y_k$ and $Y=D\cup\{x_0\}$.  Then $\mathbf{D}$ is a discrete cuf subspace of $\mathbf{X}$, the subspace $\mathbf{Y}$ of $\mathbf{X}$ is compact and  $x_0$ is a unique accumulation point of $\mathbf{Y}$. Hence $\mathbf{Y}$ is a copy of $\mathbf{D}(\infty)$. \smallskip

(ii) If $\mathbf{IDI}$ holds or $X$ is well-orderable, then we can fix a disjoint family $\{F_n: n\in\omega\}$ of singletons of $X$ and in much the same way, as in the proof to (i), we can find a copy of $\mathbb{N}(\infty)$ in $\mathbf{X}$. By Proposition 4.4 of \cite{ktw0}, $\mathbf{WoAm}$ implies that every first-countable limit point compact $T_1$-space is well-orderable. Hence, $\mathbf{WoAm}$ implies that $\mathbf{X}$ contains a copy of $\mathbb{N}(\infty)$. 

Now, let us assume $\mathbf{BPI}$. Let $x_0$ be an accumulation point of $\mathbf{X}$. Without loss of generality, we may aassume that $\text{cl}_{\mathbf{X}}(U_{n+1})\neq U_n$ for every $n\in\omega$. Let $G_n=\{x_0\}\cup ( \text{cl}_{\mathbf{X}}(U_n)\setminus U_{n+1})$ for every $n\in\omega$. Then the subspaces $\mathbf{G}_n$ of $\mathbf{X}$ are compact.  By Theorem \ref{BPIeq}(ii), the product $\mathbf{G}=\prod\limits_{n\in\omega}\mathbf{G}_n$ is compact. Therefore, since the family  $\mathcal{G}=\{\pi_n^{-1}(\text{cl}_{\mathbf{X}}(U_n)\setminus U_{n+1}): n\in\omega\}$ has the finite intersection property and consists of closed subsets of the compact space $\mathbf{G}$, there exists $f\in\bigcap\limits_{n\in\omega}\pi_n^{-1}(\text{cl}_{\mathbf{X}}(U_n)\setminus U_{n+1})$. Then the subspace $\{x_0\}\cup\{f(2n): n\in\omega\}$ of $\mathbf{X}$ is a copy of $\mathbb{N}(\infty)$. This completes the proof to (ii).

(iii) This can be deduced from the proofs to (i) and (ii). 
\end{proof}

\begin{corollary}
\label{s4c03} 
$(\mathbf{ZF})$
\begin{enumerate}
\item[(i)] Each of $\mathbf{IQDI}$, $\mathbf{WoAm}$ and $\mathbf{BPI}$ implies that every infinite compact Hausdorff first-countable space contains an infinite metrizable compact scattered subspace.
\item[(ii)] Every infinite compact Hausdorff first-countable space which has an infinite cuf subset contains an infinite compact metrizable scattered subspace.
\end{enumerate} 
\end{corollary}

Let us recall a few known facts about the following permutation models in \cite{hr}: the Basic Fraenkel Model $\mathcal{N}1$, the Second Fraenkel Model $\mathcal{N}2$ and the Mostowski Linearly Ordered Model $\mathcal{N}3$. 

It is known that $\mathbf{WoAm}$ is true in $\mathcal{N}1$, $\mathbf{IWDI}$ (and hence the stronger $\mathbf{IQDI}$, which is implied by $\mathbf{CMC}$) is false in both $\mathcal{N}1$ and $\mathcal{N}3$, $\mathbf{CAC}_{fin}$ is true in $\mathcal{N}1$ but it is false in $\mathcal{N}2$, $\mathbf{MC}$ is true in $\mathcal{N}2$, and $\mathbf{BPI}$ (and hence $\mathbf{CAC}_{fin}$) is true in $\mathcal{N}3$.  All this, taken  together with Theorem \ref{s4t2} and its proof, implies the following theorem:

\begin{theorem}
\label{s4t04}
\begin{enumerate}
\item[(i)] It is true in $\mathcal{N}1$ that every first-countable limit point compact $T_1$-space is well-orderable. 
\item[(ii)] The sentence ``every infinite first-countable Hausdorff compact space contains a copy of $\mathbb{N}(\infty)$'' is true in $\mathcal{N}1$ and in $\mathcal{N}3$.
\item[(iii)] The sentence ``every infinite first-countable Hausdorff compact space contains an infinite metrizable compact scattered subspace'' is true in $\mathcal{N}2$.
\item[(iv)] The sentence ``every infinite first-countable Hausdorff compact space contains an infinite metrizable compact scattered subspace'' implies neither $\mathbf{CAC}_{fin}$ nor $\mathbf{IQDI}$, nor $\mathbf{CMC}$ in $\mathbf{ZFA}$.
\end{enumerate}
\end{theorem}

\begin{remark}
\label{s4r05}
In \cite{ktw0}, it was shown that $\mathbf{M}(C, S)\wedge\neg\mathbf{IDI}$ has a $\mathbf{ZF}$-model. Using similar arguments, one can show that $\mathbf{M}(C, S)\wedge\neg\mathbf{IQDI}$ also has a $\mathbf{ZF}$-model.  This, together with Corollary \ref{s4c03}, proves that  the sentence ``every infinite compact metrizable space contains an infinite compact scattered subspace'' does not imply  $\mathbf{IQDI}$ in $\mathbf{ZF}$. This can be also deduced from Theorem \ref{s4t04}.

(ii) Suppose that $\{A_n: n\in\omega\}$ is a disjoint family of non-empty finite sets which does not have any partial choice function. Let $D=\bigcup\limits_{n\in\omega}A_n$ and $\mathbf{D}=\langle D, \mathcal{P}(D)\rangle$. Then $\mathbf{D}(\infty)$ is a compact metrizable scattered space which does not contain a copy of $\mathbb{N}(\infty)$.
\end{remark}

Let us finish our results with the following proposition:

\begin{proposition}
\label{s4p06}
$(\mathbf{ZF})$
A metrizable space  $\mathbf{X}$ contains an infinite compact scattered subspace if and only if $\mathbf{X}$ contains a copy of $\mathbf{D}(\infty)$ for some infinite discrete cuf space $\mathbf{D}$. 
\end{proposition}

\begin{proof}
$(\rightarrow)$ Suppose that $\mathbf{X}$ is a metrizable space which has an infinite compact scattered subspace $\mathbf{Y}$. Then $\Iso(Y)$ is a dense discrete subspace of $\mathbf{Y}$, so $\mathbf{Y}$ is a metrizable compactification of the discrete subspace $\Iso(Y)$ of $\mathbf{Y}$. We deduce from Proposition \ref{discommet}(ii) that $\Iso(Y)$ is an infinite cuf set. Thus, by Theorem \ref{s4t2}(iii), $\mathbf{Y}$ contains a copy of $\mathbf{D}(\infty)$ for some infinite discrete cuf space $\mathbf{D}$.

$(\leftarrow)$ This is trivial because the one point Hausdorff compactification of an infinite discrete space is a scattered space.
\end{proof}

\section{A shortlist of new open problems}
\label{s6}
\begin{problem}
Find, if possible, a $\mathbf{ZF}$-model for $\mathbf{NAS}\wedge\neg\mathbf{INSHC}$.
\end{problem}

\begin{problem} Find, if possible, a $\mathbf{ZF}$-model for $\mathbf{INSHC}\wedge\neg\mathbf{BPI}\wedge\neg\mathbf{IDFBI}$.
\end{problem}

\begin{problem}
Make a list of the forms from \cite{hr} that are true in the model $\mathcal{N}$ constructed in the proof to Theorem \ref{thm:3}.
\end{problem}

\begin{problem}
Is it provable in $\mathbf{ZF}$ that every non-empty dense-in-itself compact metrizable space contains an infinite compact scattered subspace? 
\end{problem}

\begin{problem} Is it provable in $\mathbf{ZF}$ that if $\mathbf{X}$ is a compact Hausdorff non-scattered weakly Loeb space which has a cuf base, then $|\mathbb{R}|\leq|X|$?
\end{problem}

\section*{Declarations}
The authors declare no conflict of interest and no specific financial support for this work.

\end{document}